\pdfoutput=1
\documentclass[11pt]{article}

\newcommand{\ManuscriptTitle}{Empirical Bernstein in smooth Banach spaces}

\newcommand{\FirstAuthorGiven}{Diego}
\newcommand{\FirstAuthorFamily}{Martinez-Taboada}
\newcommand{\FirstAuthorEmail}{diegomar@andrew.cmu.edu}
\newcommand{\SecondAuthorGiven}{Aaditya}
\newcommand{\SecondAuthorFamily}{Ramdas}
\newcommand{\SecondAuthorEmail}{aramdas@andrew.cmu.edu}
\newcommand{\FirstAuthorAffiliation}{Department of Statistics \& Data Science, Carnegie Mellon University}
\newcommand{\SecondAuthorAffiliation}{Department of Statistics \& Data Science and Machine Learning Department, Carnegie Mellon University}

\newcommand{\ManuscriptAbstract}{Existing concentration bounds for bounded vector-valued random variables include extensions of the scalar Hoeffding and Bernstein inequalities. While the latter is typically tighter, it requires knowing a bound on the variance of the random variables. We derive a new vector-valued empirical Bernstein inequality, which makes use of an empirical estimator of the variance instead of the true variance. The bound holds in 2-smooth separable Banach spaces, which include finite dimensional Euclidean spaces and separable Hilbert spaces. The resulting confidence sets are instantiated for both the batch setting (where the sample size is fixed) and the sequential setting (where the sample size is a stopping time). The confidence set width asymptotically exactly matches that achieved by Bernstein in the leading term.}

\newcommand{\KeywordOne}{empirical Bernstein}
\newcommand{\KeywordTwo}{concentration inequalities}
\newcommand{\KeywordThree}{2-smooth Banach spaces}

\newcommand{\ManuscriptFunding}{The first author acknowledges that the project that gave rise to these results received the support of a fellowship from `la Caixa' Foundation (ID 100010434). The fellowship code is LCF/BQ/EU22/11930075.\par\medskip The second author was supported by NSF Grant DMS-2310718.}

\usepackage[margin=1in]{geometry}
\usepackage{amsthm,amsmath,amsfonts,amssymb}
\usepackage[numbers,sort&compress]{natbib}
\usepackage[colorlinks,citecolor=blue,urlcolor=blue,linkcolor=blue]{hyperref}
\usepackage{graphicx}
\usepackage{xcolor}
\usepackage{dsfont}
\usepackage{microtype}
\usepackage{enumitem}

\allowdisplaybreaks

\newlist{longlist}{enumerate}{1}
\setlist[longlist]{label=(\roman*),leftmargin=*,topsep=3pt,itemsep=1pt}

\theoremstyle{plain}

\newtheorem{theorem}{Theorem}
\newtheorem{lemma}{Lemma}
\newtheorem{proposition}{Proposition}
\newtheorem{corollary}{Corollary}

\theoremstyle{definition}
\newtheorem{definition}{Definition}

\newtheorem{fact}{Fact}

\newcommand{\Pb}{\mathbb{P}}
\newcommand{\Nb}{\mathbb{N}}
\newcommand{\R}{\mathbb{R}}
\newcommand{\E}{\mathbb{E}}
\newcommand{\X}{\mathcal{X}}

\newcommand{\F}{\mathcal{F}}

\newcommand{\lp}{\left(}
\newcommand{\rp}{\right)}
\newcommand{\lc}{\left\{}
\newcommand{\rc}{\right\}}
\newcommand{\lb}{\left[}
\newcommand{\rb}{\right]}
\newcommand{\rba}{\right|}
\newcommand{\lba}{\left|}
\newcommand{\rdba}{\right\|}
\newcommand{\ldba}{\left\|}
\newcommand{\ra}{\right\rangle}
\newcommand{\la}{\left\langle}

\newcommand{\pe}{\psi_E}
\newcommand{\pel}{\psi_E(\lambda)}
\newcommand{\ind}{\mathds{1}}
\newcommand{\norm}[1]{\left\lVert #1\right\rVert}

\newcommand{\changemark}[1]{}
\newcommand{\changed}[2]{#2}
\newcommand{\changednomark}[1]{#1}
\newcommand{\deletedchange}[2]{}
\newenvironment{changedblock}[1]{\begingroup}{\endgroup}

\title{\ManuscriptTitle}
\author{
  \FirstAuthorGiven\ \FirstAuthorFamily\thanks{\FirstAuthorAffiliation. Email: \texttt{\FirstAuthorEmail}.}
  \and
  \SecondAuthorGiven\ \SecondAuthorFamily\thanks{\SecondAuthorAffiliation. Email: \texttt{\SecondAuthorEmail}.}
}
\date{}

\begin{document}
\maketitle

\begin{abstract}
\ManuscriptAbstract
\end{abstract}


\section{Introduction}

The concentration phenomenon of averages of random variables composes one of the cornerstones of statistics. In the case of bounded (scalar) random variables, \cite{hoeffding1963probability} presented a celebrated concentration inequality based on upper bounding the moment generating function. This is commonly referred to as Hoeffding's inequality, and it solely relies on the bounds on the random variables. This bound may be sharpened if one knows the variance of these random variables, giving way to the so-called Bernstein and Bennett inequalities \cite{bernstein_theory_1927, bennett1962probability}. While these inequalities can be inverted to yield tighter confidence intervals for the mean, they are not actionable in practice as they require knowledge of (a nontrivial upper bound on) the variance, which is generally unknown. This motivated the development of so-called empirical Bernstein inequalities \cite{audibert2007tuning, maurer2009empirical, waudby2024estimating}, which make use of an empirical estimator of the variance instead of the true variance. These are typically tighter than Hoeffding's because they can adapt to low variance data, and are applicable because they require no further information than the bounds on the random variables.

This paper will extend such techniques to random variables in more than one dimension. 
In a seminal contribution, \cite{pinelis1994optimum} extended Hoeffding's, Bernstein's and Bennett's inequalities to smooth Banach spaces. However, presenting a dimension-free, multivariate version of the empirical Bernstein inequality remains an open problem. 

The practical importance of providing a concentration inequality for bounded multivariate objects that adapts to their unknown variance should be readily apparent. Similarly to the scalar case, a multivariate empirical Bernstein inequality would in general provide tighter confidence sets than those given by Hoeffding's inequality. A sensible way of judging the tightness of an empirical Bernstein inequality is to check whether the dominant term of its width converges to that of the (oracle) Bernstein inequality that knows the variance. We develop such tight empirical Bernstein inequalities for smooth Banach spaces in this work.


\subsection{Our contributions}

The primary contribution of our work is to provide a tight multivariate version of the empirical Bernstein inequality (Section~\ref{section:main_result}). Seeking  generality of the results, and similarly to \cite{pinelis1994optimum}, our result holds in 2-smooth separable Banach spaces. These include any separable Hilbert space, such as the Euclidean spaces $\R^d$ with the usual inner product. The concentration inequality is dimension-free and its first order term is shown to match that of the oracle Bernstein inequality in Banach spaces, including constants (Proposition~\ref{proposition:radius_ball_iid}).

Our multivariate empirical Bernstein inequality is obtained by coupling a new supermartingale construction (Theorem~\ref{theorem:main_theorem}) with Ville's inequality.
In fact, the results are obtained as a byproduct of a sequential empirical Bernstein inequality (Corollary~\ref{cor:empirical_bernstein}), where the sample size is not fixed beforehand and observations are obtained one at a time. Consequently, our empirical Bernstein bound may be instantiated for the classical batch setting (Section~\ref{section:fixed_sample_size}), where the sample size is fixed beforehand, but also the sequential setting (Section~\ref{section:sequential}), where the sample size is a stopping time and is not fixed beforehand. 


Our arguments are inspired by the approach in \cite{pinelis1992approach, pinelis1994optimum}, but substantial differences may be noted between the two contributions. The heart of the (oracle) Pinelis' Bernstein-type inequality is to use a Taylor expansion argument, make the first derivative term disappear thanks to the martingale condition, and upper bound the second term using an immediate inequality; the result then follows with relative ease. However, our first derivative does not have zero expectation given that the empirical variance estimator introduces bias (Equation~\eqref{eq:second_term}), and the bound for the second derivative is inflated by the emergence of additional terms (Equation~\eqref{eq:second_derivative_dct}). Two strong inequalities are thus needed to combine all the pieces of the Taylor expansion (Lemma~\ref{lemma:main_ineq} and Lemma~\ref{lemma:ineq_cosh_sinh}). Importantly, the most intricate part of our proof lies not merely in establishing these inequalities, but in \textit{recognizing} that they hold and that their validity is precisely what allows the argument to go through —particularly in the case of the sharp scalar inequality of Lemma~\ref{lemma:main_ineq}.  From a higher level, the complexity of our new supermartingale constructions requires us to work with higher order terms whose effects eventually cancel out. The tighter analysis that stems from examining these higher order terms is key in establishing the result (compare the tightness of Lemma~\ref{lemma:main_ineq} with the inequality from \cite[Proposition 4.1]{fan_exponential_2015}). On top of that, a number of technical challenges arise (see Appendix~\ref{section:aux_lemmas} for the new auxiliary results that ought to be derived). We expand on these details in Section~\ref{section:main_proof}. We believe that our techniques and overall approach may serve as a foundation for new concentration inequalities, as we discuss in Section~\ref{section:summary}.

\subsection{Related Work}

Our work builds on the techniques exhibited in  \cite{pinelis1994optimum}, who presented Bernstein, Bennett, and Hoeffding type inequalities for 2-smooth separable Banach space-valued random variables. \cite{pinelis1994optimum} developed the theoretical tools from \cite{pinelis1992approach} in greater generality. A supermartingale construction underlies their concentration inequalities as well, and their results are also dimension-free. Our contribution can be seen as extending Pinelis' works to the case where the variance is unknown, but we would like to attain exactly the same limiting (leading order) bounds.

\subsubsection*{Bounded vector-valued concentration inequalities} There are other works that have studied concentration inequalities for bounded random variables in multivariate spaces. 
\cite{kallenberg1991some} gives weaker results than \cite{pinelis1994optimum}, and their method seems to be confined solely to Hilbert spaces. \cite{kohler2017sub}, Lemma 18, and \cite{gross2011recovering}, Theorem 12, also provide dimension-free vector-valued Bernstein inequalities. These are again weaker than those derived in \cite{pinelis1994optimum} and restricted to independent random variables (see Appendix~\ref{appendix:optimality_pinelis} for a comparison of the contributions). \cite{whitehouse2023time}, Theorem 6.1, presented a self-normalized, multivariate empirical Bernstein inequality; their bound relies on a covering argument, thus being dimension-dependent and limited to finite-dimensional Euclidean spaces. Similarly, \cite{chugg2023time} provided a time-uniform vector-valued empirical Bernstein inequality using a PAC-Bayes argument, with the radii of their confidence sequences depending again on the dimension of the Euclidean space and not being generalizable to infinite-dimensional spaces.


\subsubsection*{Scalar empirical Bernstein inequalities} The first scalar empirical Bernstein inequalities were presented in \cite{audibert2007tuning}, Theorem 1, \cite{maurer2009empirical}, Theorem 11, and \cite{balsubramani2016sequential}, Theorem 5.  A different concentration inequality was proposed in \cite{mhammedi2019pac}, Lemma 13, in the context of PAC-Bayesian generalization bounds. Nonetheless, their empirical performance was improved in \cite{waudby2024estimating}, Theorem 2, which extended \cite{howard2021time}, Theorem 4. In fact, the latter is the only scalar empirical Bernstein we are aware of whose first order term exactly matches that of Bernstein in the limit (as will be made precise later in this paper). The confidence sets presented in this contribution can be seen as generalizing the \cite{waudby2024estimating} (scalar) empirical Bernstein bounds to smooth Banach spaces.

\subsubsection*{Time-uniform Chernoff bound} Deriving concentration inequalities through nonnegative supermartingale constructions has recently received widespread attention as they provide, in view of Ville's inequality, probabilistic guarantees for streams of data that are continuously monitored and adaptively stopped. Our work unifies that of \cite{pinelis1994optimum} and~\cite{waudby2024estimating}; all three can be viewed as instances of the time-uniform Chernoff bound framework of \cite{howard2020time, howard2021time}. 

\subsubsection*{Gambling-based concentration} There exists a deep connection between concentration inequalities and the regret guarantee of online learning algorithms with linear losses, as elucidated in \cite{rakhlin2017equivalence}. Thus, studying concentration phenomena via \textit{gambling} games composes a growing and exciting line of research \cite{jun2019parameter, shekhar2023near, waudby2024estimating}. However, the concentration bounds may not have a closed analytical form. For example, a betting-based strategy to construct a scalar empirical Bernstein concentration inequality was presented in \cite{orabona2023tight}, Theorem 3; the confidence sets obtained from inverting such an inequality are analytically intractable, and their computation is nontrivial. While prominent online learning algorithms have been designed in Banach spaces \cite{cutkosky2018black}, the non-trivial inversion needed to obtain confidence sets is especially challenging in multivariate settings (resulting in most of the development having been tailored to scalar-valued processes). \cite{cutkosky2019combining}, Appendix A, presented an empirical Bernstein-type concentration inequality in Hilbert spaces; however, the radii of the confidence sets are established up to polylogarithmic factors. \cite{ryu2024gambling} proposed multivariate gambling-based confidence sequences  but their work is limited to finite dimensional Euclidean spaces $\R^d$ and are not closed form, making it hard to judge their limiting width, and have a computational complexity of $O(t^d)$ for verifying whether a vector belongs to the confidence sequence at time $t$. %

\subsection{Paper outline}
We organize this manuscript as follows. We present preliminary work in Section~\ref{section:background}, putting the focus on Ville's inequality and 2-smooth separable Banach spaces. The key ideas that underpin the scalar Bernstein, scalar empirical Bernstein and multivariate Bernstein inequalities are then exhibited in Section~\ref{section:scalar_ideas}. Section~\ref{section:main_result} is dedicated to the statement and implications of the main theorem of the paper, a (empirical Bernstein-type) supermartingale construction for 2-smooth separable Banach space-valued data. 
We instantiate such a result to obtain empirical Bernstein confidence intervals for the batch setting, and confidence sequences for the sequential setting. Furthermore, we compare the empirical performance of the proposed confidence sets against those obtained via Hoeffding's and Bernstein's inequalities. In Section~\ref{section:main_proof}, the proof of the main theorem of the paper is presented. We conclude with some remarks in Section~\ref{section:summary}.

\section{Preliminaries} \label{section:background}

This section contains the technical definitions that are necessary to understand our method and proof technique.

\subsection{Nonnegative supermartingales and Ville's inequality}

Let us start by presenting some essential tools. A  filtration $\F = (\F_t)_{t \geq 0}$ is a sequence of $\sigma$-algebras such that $\F_t \subseteq \F_{t+1}$, $t \geq 0$. Throughout, we take $\F_t = \sigma(X_1, \ldots, X_t)$ to be the canonical filtration, with $\F_0$ being trivial. A stochastic process $M \equiv (M_t)_{t \geq 0}$ is a sequence of random variables that are adapted to $(\F_t)_{t \geq 0}$, meaning that $M_t$ is $\F_{t}$-measurable for all $t$. $M$ is called \emph{predictable} if $M_t$ is $\F_{t-1}$-measurable for all $t$. 
An integrable stochastic process $M$ is a supermartingale if $\E[M_{t+1}| \F_t] \leq M_t$ for all $t$, and a martingale if the inequality always holds with equality. Inequalities or equalities between random variables are always interpreted to hold almost surely. Throughout, we use the shorthand $\E_{t}[\cdot] = \E[\cdot | \F_t]$. 

Nonnegative supermartingales play a central role in deriving concentration inequalities due to Ville's inequality \cite{ville1939etude}.

\begin{fact} [Ville's inequality] \label{fact:villesineq} If $M$ is a nonnegative supermartingale adapted to $\{ \F_t\}_{t \geq 0}$, then for any $x > 0$,
\begin{align*}
    \Pb(\exists t \geq 0 : M_t \geq x) \leq \frac{\E[M_0]}{x}. 
\end{align*}
\end{fact}


\subsection{2-smooth Banach spaces}

If $X_1, \ldots, X_n$ are independent, centered random vectors belonging to a Hilbert space, then 
\begin{align*}
    \E \lb \ldba \sum_{i = 1}^n X_i \rdba^2 \rb 
    =  \E \lb \sum_{i = 1}^n \ldba  X_i \rdba^2 \rb.
\end{align*}
In stark contrast, the argument does not follow in Banach spaces, due to the absence of an inner product. Thus, even the apparent basic problem of bounding the expectation of the squared norm of sums of independent random variables in terms of their sum of squared norms is not straightforward in general Banach spaces. This motivates the use of Banach spaces with the so-called Rademacher types or cotypes, which are notions of probabilistic orthogonality (see e.g. \cite{ledoux2013probability}, Chapter 9). Nonetheless, such concepts are not sufficient without the assumption of independence. For martingales,
\cite{pisier1975martingales} elucidated the relevance of 2-smooth spaces, which play the same role with respect to  vector martingales as the spaces of Rademacher type 2 do with respect to the sums of independent random vectors.

\begin{definition}
A Banach space $(\X, \| \cdot \|)$ is said to be $(2, D)$-smooth for some $D > 0$ (or, in short, 2-smooth) if 
\begin{align} \label{eq:definition_2smooth}
    \| x + y \|^2 + \|x - y \|^2 \leq 2 \|x\|^2 + 2D^2 \|y\|^2, \quad \forall (x, y) \in \X^2.
\end{align}
\end{definition}

For $\X \neq \{ 0 \}$, taking $x=0$ and $y\neq0$ implies $D \geq 1$. Throughout, we will assume that $\X \neq \lc 0 \rc$, and so $D \geq 1$. 
Any Hilbert space is trivially (2, 1)-smooth. Other examples include $L^p(T, \mathcal{A}, \nu)$ being $(2, \sqrt{p-1})$-smooth for any measure space $(T, \mathcal{A}, \nu)$ and $p \geq 2$ \cite{pinelis1994optimum}, as well as \changed{1}{Schatten} trace ideals $C^p$ being $(2, \sqrt{p-1})$-smooth for $p \geq 2$ \cite{ball2002sharp}. We refer the reader to  \cite{van2020maximal}, Section 2.2, for further examples. 

Throughout, we will work with $(2, D)$-smooth separable Banach spaces. The separability assumption is ubiquitous in the literature, as it avoids measurability issues and implies the tightness of any distribution in the Banach space \cite{ledoux2013probability}. Furthermore, separability holds in most practical situations, the most prominent examples being $\R^d$ with the usual inner product, or reproducing kernel Hilbert spaces (RKHS's) associated to continuous reproducing kernels on separable domains \cite{steinwart2008support}. 

\section{Hoeffding and (empirical) Bernstein concentration} \label{section:scalar_ideas}

This section recaps the statements of the relevant existing concentration work in the scalar and Banach space settings, in order to contextualize the statement of our main result in the next section. We also briefly explain their proof at a high level, as a precursor to the long and technical proof of our result.

\subsection{Scalar Hoeffding's and Bernstein's inequalities} \label{section:scalar_bernstein}


Fix $n$ and let $X_1, \ldots, X_n$ be real random variables such that  $\E_{t-1}[X_t] = \mu$ and $|X_t - \mu| \leq B$. The scalar Hoeffding inequality \cite{hoeffding1963probability} and Bernstein inequality (see e.g. \cite{vershynin2018high}, Chapter 2) respectively imply that, for all $r > 0$,
\begin{align*}
\text{(Hoeffding)}\quad &\Pb \lp  \lba \sum_{i = 1}^n (X_i- \mu) \rba \geq r \rp \leq 2\exp \lp - \frac{r^2}{2nB^2} \rp,
 \\
 \text{(Bernstein)}\quad  &\Pb \lp  \lba \sum_{i = 1}^n (X_i- \mu) \rba \geq r \rp \leq 2\exp \lp - \frac{\frac{r^2}{2}}{\ldba \sum_{i = 1}^n  \E_{i-1}[(X_i-\mu)^2] \rdba_\infty  + r\frac{B}{3} }\rp,
\end{align*}
where $\ldba \cdot \rdba_\infty$ denotes the essential supremum of a random variable. Generally, Bernstein's inequality is instantiated for iid random variables, in which case $\ldba \E_{i-1}[(X_i-\mu)^2] \rdba_\infty$ is simply their variance. 

Bernstein's inequality relies on upper bounding the probability of the events 
\begin{align*}
    \sum_{i = 1}^n (X_i- \mu) \geq r, \quad \sum_{i = 1}^n (X_i- \mu)\leq -r,
\end{align*}
and subsequently applying the union bound. Probabilistic guarantees for each of the events may be obtained from upper bounding the moment generating function by the following expression:
\begin{align} 
    \label{eq:bernstein_scalar} \exp \lp \frac{1}{B^2}\ldba\sum_{i=1}^n \E_{i-1}[(X_i-\mu)^2]\rdba_\infty\sum_{k  \geq 2} \frac{\lambda^k B^{k} }{k!}\rp.
\end{align}
 Applying Markov's inequality using this upper bound and optimizing for $\lambda$ eventually recovers the inequality. Note that an analogous (infinite dimensional) vector-valued inequality does not immediately follow from the scalar case, as random vectors may take values in infinite directions, making a union bound argument useless.




\subsection{Pinelis' Hoeffding and Bernstein inequalities in smooth Banach spaces}

Let now $X_1, \ldots, X_n$ belong to a $(2, D)$-smooth separable Banach space such that $\E_{t-1} \lb X_t \rb = \mu$ and $\ldba X_t - \mu \rdba \leq B$. From the results of \cite{pinelis1994optimum}, Theorem 3.1, one can deduce (see Appendix~\ref{appendix:pinelis_derivation})\footnote{The Bernstein-type concentration bound presented here can be derived from \cite{pinelis1994optimum}, Theorem 3.3 (taking $\Gamma = B / 3$), as well as from \cite{pinelis1994optimum}, Theorem 3.4. We show how to obtain it from the latter in Appendix~\ref{appendix:pinelis_derivation}.}   that for all $r > 0$,
\begin{align*}
&\Pb \lp  \ldba \sum_{i = 1}^n (X_i- \mu) \rdba \geq r \rp \leq 2\exp \lp - \frac{r^2}{2nD^2B^2}\rp,
    \\&\Pb \lp  \ldba \sum_{i = 1}^n (X_i- \mu) \rdba \geq r \rp \leq 2\exp \lp - \frac{\frac{r^2}{2}}{ D^2\ldba\sum_{i = 1}^n  \E_{i-1}\lb\ldba X_i-\mu\rdba^2\rb \rdba_\infty  + r\frac{B}{3} }\rp.
\end{align*}

In order to prove the Bernstein-type concentration bound, \cite{pinelis1994optimum} proposed to study the process $\cosh \lp \lambda \| M_t \|\rp$, where $M_t = \sum_{i = 1}^t (X_i- \mu)$. Ideally, one would like to show that $R_t\cosh \lp \lambda \| M_t \|\rp$ is a supermartingale, for an appropriate $R_t$. We provide a short proof below when $\| \cdot \|$ is twice Fréchet differentiable. Let $\varphi_t(\beta) = \cosh \lp \lambda \| M_{t-1} + \beta (X_t-\mu) \| \rp$, and note that $\E_{t-1} \lb \cosh \lambda\|M_t\|\rb = \E_{t-1} \lb \varphi_t(1)\rb$.  By a Taylor expansion argument, 
\begin{align*}
    \E_{t-1} \lb \varphi_t(1)\rb &= \E_{t-1} \lb \varphi_t(0) + \varphi'_t(0) + \int_0^1 (1 - \beta)\varphi''(\beta)d\beta\rb
    \\&= \cosh\lp \lambda \ldba M_{t-1} \rdba \rp + \lambda \sinh \lp \lambda \| M_{t-1} \|\rp  \ldba M_{t-1} \rdba'\lp \E_{t-1} \lb X_t - \mu \rb \rp  
    \\&\quad+ \E_{t-1} \lb \int_0^1  (1 - \beta) \varphi''(\beta)d\beta\rb
    \\&= \cosh\lp \lambda \ldba M_{t-1} \rdba \rp + 0  + \E_{t-1} \lb \int_0^1 (1 - \beta)\varphi''(\beta)d\beta\rb
    \\&\stackrel{}{\leq} \cosh\lp \lambda \ldba M_{t-1} \rdba \rp + \cosh\lp \lambda \ldba M_{t-1} \rdba \rp D^2 \E_{t-1} \lb \sum_{k \geq 2} \frac{\lp \lambda \ldba X_t -\mu\rdba\rp^k}{k!} \rb
    \\&\stackrel{}{=} \cosh\lp \lambda \ldba M_{t-1} \rdba \rp \lp 1 + D^2\E_{t-1} \lb \sum_{k \geq 2} \frac{\lp \lambda \ldba X_t - \mu \rdba\rp^k}{k!} \rb\rp
    \\&\stackrel{}{\leq} \cosh\lp \lambda \ldba M_{t-1} \rdba \rp \exp\lp D^2\E_{t-1} \lb \sum_{k \geq 2} \frac{\lp \lambda \ldba X_t - \mu \rdba\rp^k}{k!} \rb\rp ,
\end{align*}
where $\| \cdot \|'$ denotes the Fréchet derivative of the function $\| \cdot \|$, and \cite{pinelis1994optimum} obtains the (first) key inequality by combining the properties of the $\cosh$ function and the derivatives of the norm (which is precisely the motivation for using the $\cosh$ function over the $\exp$ function). It follows that
\begin{align*}
     \cosh \lp \lambda \| M_n \|\rp \underbrace{\exp\lp -\frac{D^2}{B^2} \ldba \sum_{i = 1}^n  \E_{i-1} \lb \ldba X_i-\mu \rdba ^2\rb \rdba_\infty \sum_{k \geq 2} \frac{\lp \lambda B \rp^k}{k!} \rp}_{:= R_n},
\end{align*}
has expectation at most $1$. Applying Markov's inequality to $R_n \cosh \lp \lambda \| M_n \|\rp$ and optimizing for $\lambda$ eventually gives way to the stated inequalities; applying Ville's inequality would give a time-uniform bound. 

To summarize, like the scalar case, the heart of this proof is to provide upper bounds which only require second and higher moments of the random variables. In the scalar case, this may be obtained via a union bound; in the vector-valued case, it stems from the choice of the $\cosh$ function over the $\exp$ function, alongside a Taylor expansion argument. 

Importantly, however, $\| \cdot \|$ need not be Fréchet differentiable (for instance, the absolute value function is not differentiable at $0$). In a sophisticated and technical proof, \cite{pinelis1994optimum} showed that smooth, finite dimensional approximations of $R_t \cosh \lp \lambda \| M_t \|\rp$ are indeed supermartingales, with the previous arguments being applicable to such approximations. Our aim above is to summarize the key intuition behind his proof, because the proof of our main result will also rely on proving supermartingale structures of smooth, finite dimensional approximations of the original process.



\subsection{Waudby-Smith and Ramdas' scalar empirical Bernstein inequality}

As mentioned earlier, there exist several empirical Bernstein inequalities \cite{audibert2007tuning, maurer2009empirical}. We focus here on the concentration inequality provided by \cite{howard2021time} and \cite{waudby2024estimating}, which the latter paper shows yields the most accurate empirical confidence intervals and comes with better theoretical guarantees in that it recovers exactly the first order term of Bernstein's inequality, including constants, as summarized below. In stark contrast to the oracle Bernstein inequality, which requires knowledge of the variance of the random variables, it exploits an empirical estimator of the variance instead. Thus, it yields confidence intervals (and confidence sequences) that may be used in practice and are tighter than those given by Hoeffding's inequality.

Let $X_1, \ldots, X_n$ be real random variables such that $\E_{t-1}[X_t] = \mu$ and $|X_t| \leq B$. The scalar empirical Bernstein proposed in \cite{waudby2024estimating} establishes that\footnote{\cite{waudby2024estimating} stated  results for $X_t \in [0,1]$; we adapted them to $(X_t - \mu) \in [-B, B]$.}
\begin{align*}
    \Pb \lp \lba \sum_{i = 1}^n \lambda_i \frac{X_i - \mu}{2B} \rba   - \sum_{i = 1}^n \frac{\psi_E(\lambda_i)}{(2B)^2}  \lp X_i - \hat\mu_{i-1} \rp^2 \geq \log\lp\frac{2}{\alpha}\rp\rp \leq \alpha, \quad \forall \alpha \in (0, 1],
\end{align*}
where $\psi_E(\lambda) = - \log ( 1 - \lambda) - \lambda$, $\hat\mu_{i}$ is any predictable estimator of $\mu$, and $\lambda_i \in (0,1)$ is any predictable sequence.
Letting $r = \log(2/\alpha)$ makes the right hand side equal to $2\exp(-r)$, facilitating easier comparison to the earlier bounds.
Thus, 
\begin{align*}
    \lp \frac{\sum_{i = 1}^n \lambda_i X_i}{\sum_{i = 1}^n \lambda_i} \pm  \frac{(2B)\log\lp\frac{2}{\alpha}\rp + \sum_{i = 1}^n \frac{\psi_E(\lambda_i)}{2B}  \lp X_i - \hat\mu_{i-1} \rp^2}{\sum_{i = 1}^n \lambda_i}\rp
\end{align*}
constitutes a $(1-\alpha)$-confidence interval for $\mu$. For iid observations, and appropriate choices of $\lambda_i$ and $\hat\mu_i$,
\cite{waudby2024estimating} showed that
\begin{align*}
    \sqrt{n}\lp\frac{(2B)\log\lp\frac{2}{\alpha}\rp + \sum_{i = 1}^n \frac{\psi_E(\lambda_i)}{2B}  \lp X_i - \hat\mu_{i-1} \rp^2}{\sum_{i = 1}^n \lambda_i} \rp \stackrel{a.s.}{\to} \sigma \sqrt{2 \log \lp \frac{2}{\alpha}\rp}.
\end{align*}

The Bernstein inequality stated in~\eqref{eq:bernstein_scalar} yields confidence intervals whose width  is (see e.g. Appendix~\ref{appendix:radius_pinelis} for a derivation)
\begin{align*}
    \sigma\sqrt{\frac{2\log(2/\alpha)}{n}} + \frac{2B\log(2/\alpha)}{3n}.
\end{align*}

That is, the first-order asymptotic width (i.e., scaled by $\sqrt{n}$) of the confidence intervals given by \changed{2}{\cite{waudby2024estimating}} empirical Bernstein  is equal to those provided by Bernstein inequality, and it is the only closed-form interval we are aware of with this property.

The proof of the scalar empirical Bernstein proposed in \cite{waudby2024estimating} is based on nonnegative supermartingales and a simple union bound for the upper and lower inequalities, and it does not easily generalize to vector-valued random variables. The supermartingale property is proven by combining an inequality  exhibited in \cite{fan_exponential_2015}, Proposition 4.1, with an additional trick from \cite{howard2021time}, Section A.8.

\section{Empirical Bernstein in 2-Smooth Banach spaces} \label{section:main_result}

We are now ready to present Theorem~\ref{theorem:main_theorem}, the main contribution of this work.  It is a generalization of the nonnegative supermartingale construction introduced in \cite{howard2021time} and \cite{waudby2024estimating} to 2-smooth Banach spaces.



\begin{theorem}
\label{theorem:main_theorem}
Let $X_1, X_2, \ldots$ be random variables in a $(2, D)$-smooth Banach space such that
    \begin{align*}
    \E_{t}[X_{t+1}] = \mu, \quad \| X_t \| \leq B, \quad \forall t \geq 1,
    \end{align*}
Let $(\lambda_t)_{t \geq 1} \in (0, 0.8]^{\Nb}$ be any predictable sequence. Define $\psi_E(\lambda) = - \log ( 1 - \lambda) - \lambda$, and
\begin{align*}
    \bar\mu_0 = 0, \quad \bar\mu_{i} = \frac{\sum_{j = 1}^i \lambda_j X_j}{ \sum_{j = 1}^i \lambda_j}, \quad i \geq 1.
\end{align*}
Then, the process 
    \begin{align} \label{eq:main_process} 
        S_t = \cosh \lp \ldba \sum_{i = 1}^t \lambda_i \frac{X_i- \mu}{4BD} \rdba \rp \exp \lp - \sum_{i = 1}^t \frac{\psi_E(\lambda_i)}{(4B)^2} \| X_i - \bar\mu_{i-1} \|^2  \rp
    \end{align}
is a nonnegative supermartingale. 
\end{theorem}

\deletedchange{3}{The theorem remains equally valid under the assumption $\|X_t-\mu\|\leq B$.}We defer its proof to Section \ref{section:main_proof}; it nontrivially combines ideas from \cite{howard2021time} and \cite{waudby2024estimating} with the sophisticated techniques from \cite{pinelis1994optimum}. The process itself strongly resembles that of \cite{waudby2024estimating}, but replacing the $\exp$ function by the $\cosh$ function, so that the martingale property of the sum of random variables can be established without applying the union bound (in a similar spirit to  \cite{pinelis1994optimum} discussed earlier). Nonetheless, a variety of new technical challenges arise in this proof, due to the complexity of the novel supermartingale construction. For example, the inequality presented in Lemma~\ref{lemma:main_ineq} is quite subtle (which may be noted by the need of studying a $7$-th degree Taylor coefficient), and is key in the result.



The nonnegative supermartingale established in Theorem~\ref{theorem:main_theorem} gives way to anytime concentration bounds in $(2, D)$-smooth Banach spaces due to Ville's inequality. The following corollary exhibits the subsequent empirical Bernstein inequality; we provide its proof in Appendix~\ref{section:proof_main_corollary}. 

\begin{corollary} [The empirical Bernstein inequality for $(2, D)$-smooth Banach spaces] \label{cor:empirical_bernstein} 
    Let $X_1, X_2, \ldots$ belong to a $(2, D)$-smooth Banach space. Under the assumptions of Theorem~\ref{theorem:main_theorem}, with probability $1 - \alpha$ for any $\alpha \in (0, 1)$, and simultaneously for all $t \geq 1$,
    \begin{align*}
        \ldba \frac{\sum_{i=1}^{t} \lambda_i X_i}{\sum_{i=1}^{t} \lambda_i} - \mu \rdba \leq D\frac{\frac{1}{4B} \sum_{i=1}^t \psi_E(\lambda_i) \ldba X_i - \bar\mu_{i-1} \rdba^2 + 4B\log \frac{2}{\alpha}}{\sum_{i=1}^{t} \lambda_i}.
    \end{align*}
\end{corollary}

The choice of appropriate sequences $(\lambda_t)$ is key in the tightness of the confidence sequences. We study two different scenarios separately. First, we propose predictable $(\lambda_t)$ for the (more classical) batch setting, where the sample size is fixed from the beginning.  Second, we analyze the sequential setting, where the sample size is not fixed beforehand and observations are made one at a time. In either case, the proposed predictable sequence $(\lambda_t)$ scales inversely with the square roots of the sequential empirical variance and sample size (up to logarithmic factors). Note that the scale of $(\lambda_t)$ does not change the center of the confidence ball, but it significantly influences its radius. 

\subsection{Fixed sample size empirical Bernstein confidence sets} \label{section:fixed_sample_size}

We first revisit the more classical problem of having a sample $X_1, \ldots, X_n$ for a fixed sample size $n$. Generally, the observations are assumed to be independent and identically distributed. Nonetheless, it suffices that the observations attain the (substantially more lenient) assumptions from Theorem~\ref{theorem:main_theorem}. We propose to instantiate Corollary~\ref{cor:empirical_bernstein} with $(\lambda_t) = (\lambda_t^{\text{CI}})$, where
\begin{align*}  
    \lambda_t^{\text{CI}} := \min\lp\sqrt{\frac{2(4B)^2 \log (\frac{2}{\alpha})}{\hat\sigma_{t-1}^2n}}, c_1 \rp, \quad \hat\sigma_{t}^2 = \frac{c_2B^2 + \sum_{i = 1}^t \| X_i - \bar\mu_{i-1} \|^2}{t+1}, 
\end{align*}
for some $c_1 \in (0,0.8]$ and $c_2 \in \changed{29}{(0, 1]}$. The parameter $c_1$ controls that $\lambda_t^{\text{CI}}$ does not explode in the event of small $\hat\sigma_{t-1}^2$, and parameter $c_2 > 0$ avoids extremely small values of $\hat\sigma_{t-1}^2$ for small $t$ (note that, for small $t$, $\hat\sigma_{t-1}^2$ may be substantially small with considerable probability). Reasonable default choices are $c_1 = 1/2$ or $c_1 = 3/4$, and $c_2 = 1/4$.

For independent and identically distributed random variables, the following theorem establishes the limiting radius of the confidence sets given by empirical Bernstein with the previous choice $(\lambda_t) = (\lambda_t^{\text{CI}})$. Its proof may be found in Appendix ~\ref{proof:radius_ball_iid}.
\begin{proposition} \label{proposition:radius_ball_iid}
    Fix $n \in \Nb$, and let $X_1, \ldots, X_n \stackrel{iid}{\sim} X$. Denote $\E\lb X\rb = \mu$, and let \changed{4}{\mbox{$\sigma>0$}} be defined by $\sigma^2 = \E\lb \| X - \mu \|^2\rb$. The radius of the confidence ball multiplied by $\sqrt n$ has an asymptotic limit:
    \begin{align*}
        \sqrt{n}\lp D\frac{\frac{1}{4B} \sum_{i=1}^n \psi_E(\lambda_i^{\text{CI}}) \ldba X_i - \bar\mu_{i-1} \rdba^2 + 4B\log \frac{2}{\alpha}}{\sum_{i=1}^{\changed{5}{n}} \lambda_i^{\text{CI}}}\rp \stackrel{a.s.}{\to} \sigma D\sqrt{2\log\lp\frac{2}{\alpha} \rp}.
    \end{align*}
\end{proposition}

As exhibited in Appendix~\ref{appendix:radius_pinelis}, this is precisely the limiting radius of the oracle Bernstein confidence sets derived from \cite{pinelis1994optimum}, Theorem 3.1, where $\sigma^2$ is known. Further, the limiting radius given by Proposition~\ref{proposition:radius_ball_iid} coincides with the limiting width of the scalar empirical Bernstein confidence sets given in \cite{waudby2024estimating} multiplied times $D$, which dictates the smoothness of the problem. As noted by \cite{pinelis1994optimum}, the (lack of) smoothness parameter $D$ has the effect of inflating the variance by $D^2$.

\subsection{Anytime valid empirical Bernstein confidence sequences} \label{section:sequential}

Let now $X_1, X_2, \ldots$ be a stream of data, where observations are made one at a time. Corollary~\ref{cor:empirical_bernstein} opens the door to  constructing confidence sequences, \changed{6}{that is},
sequences of confidence intervals that are uniformly valid over an unbounded time horizon. This  allows for
conducting safe anytime-valid inference, given that its probabilistic guarantees hold simultaneously for all $t$. Consequently, a practitioner may \changed{7}{peek} at the confidence sequence at any time to make statistical decisions that remain valid. 

In this scenario, we propose to instantiate Corollary~\ref{cor:empirical_bernstein} with $(\lambda_t) = (\lambda_t^{\text{CS}})$, where
\begin{align*}  
    \lambda_t^{\text{CS}} := \min\lp\sqrt{\frac{2(4B)^2 \log (\frac{2}{\alpha})}{\hat\sigma_{t-1}^2t\log(1+t)}}, c_1 \rp, \quad \hat\sigma_{t}^2 = \frac{c_2B^2 + \sum_{i = 1}^t \| X_i - \bar\mu_{i-1} \|^2}{t+1}, 
\end{align*}
for some $c_1 \in (0,0.8]$ and $c_2 \in \changed{30}{(0, 1]}$, with again reasonable default choices being $c_1 = 1/2$ or $c_1 = 3/4$,  and $c_2 = 1/4$.

Note that, in comparison to the batch setting, there is an extra $\log(1+t)$ factor. This is motivated by the fact that $\lambda_i \asymp 1 / \sqrt{i \log i}$ implies that the width scales as $\sqrt{\log t / t}$ \cite{waudby2024estimating}, which is the width of the conjugate mixture Hoeffding confidence sequence, as exhibited in \cite{howard2021time}, Proposition 2. While this width is greater than that obtained by the law of the iterated logarithm (LIL) in the limit, mixture boundaries achieve good empirical performance 
 \cite{howard2021time}. 
 
 Nonetheless, it is also possible to derive an upper LIL from Theorem \ref{theorem:main_theorem}.
 \begin{corollary} \label{corollary:lil}
     Denote $M_t = \sum_{i = 1}^t \frac{X_i- \mu}{D}$ and
     $V_t = \sum_{i = 1}^t \| X_i - \bar\mu_{i-1} \|^2 $. Under the assumptions of Theorem~\ref{theorem:main_theorem}, 
     \begin{align*}
         \limsup_{t \to\infty} \frac{\ldba M_t \rdba}{\sqrt{2V_t \log\log V_t}} \leq 1 \quad \text{ on } \lc \sup_t V_t = \infty \rc.
     \end{align*}
 \end{corollary}
 Corollary~\ref{corollary:lil} provides an asymptotic result, but
 finite LIL bounds \cite{darling1967iterated, robbins1968iterated} may as well be derived from Theorem \ref{theorem:main_theorem}. That is, confidence sequences that scale as $O\lp \sqrt{V_t \log\log V_t}\rp$ in the finite sample regime can also be obtained. As an illustrative example, the following corollary exhibits a closed-form finite LIL bound for $\alpha = 0.05$ and $B=1/4$.

 \begin{corollary} \label{corollary:finite_lil}
     Consider the case $\alpha = 0.05$ and $B=1/4$. The set composed of the vectors $x$ such that
     \begin{align} \label{eq:finite_lil_bound}
       \ldba x - \frac{\sum_{i = 1}^t X_i}{t} \rdba \leq D \frac{1.7\sqrt{\lp V_t \vee 1 \rp\lp\log\log (2(V_t \vee 1)) + 3.8\rp} + 3.4 \log\log(2(V_t \vee 1)) + 13}{t} 
  \end{align}
  constitutes a $95\%$-confidence sequence for $\mu$.
 \end{corollary}

A detailed derivation of finite LIL bounds for arbitrary values of $\alpha$ and $B$, as well as short proofs of Corollary~\ref{corollary:lil} and Corollary~\ref{corollary:finite_lil}, are provided in Appendix~\ref{appendix:lil}. Both LIL bounds follow from  the results exhibited in \cite{howard2021time}, alongside the nonnegative supermartingale construction from Theorem~\ref{theorem:main_theorem}.

\subsection{Empirical examples} \label{section:experiments}

We now run simple experiments to visualize the empirical performance of methods discussed in previous subsections. The code can be found in the Supplementary Material \cite{supp_data}.

\subsubsection*{Finite dimensional data} Figure~\ref{fig:ci} exhibits the empirical radii of the Hoeffding-type sets with our proposed empirical Bernstein confidence sets, using the tuning in Proposition~\ref{proposition:radius_ball_iid}. In both plots, the observations are iid and lie in $[-1,1]^5$, each of the components being drawn independently from  (I) a Rademacher distribution (left) or (II) a uniform distribution on $[-1, 1]$ (right). 
  In both cases, the observations are bounded by $B = \sqrt{5}$. 
  The left plot displays typical behavior of empirical Bernstein confidence sets for an extreme (high variance) distribution like the Rademacher case, where the variance is equal to $B^2$:   the empirical Bernstein bound is slightly looser than the other two at small sample sizes. In the low variance case (right plot), the empirical Bernstein confidence sets are significantly tighter than the Hoeffding bound, and only slightly looser than the (oracle) Bernstein bound. 
  These are behaviors that one observes already in the scalar case, but it is heartening that they carry forward as expected to the vector case.

\begin{figure}[t] 
    \center \includegraphics[width=\textwidth]{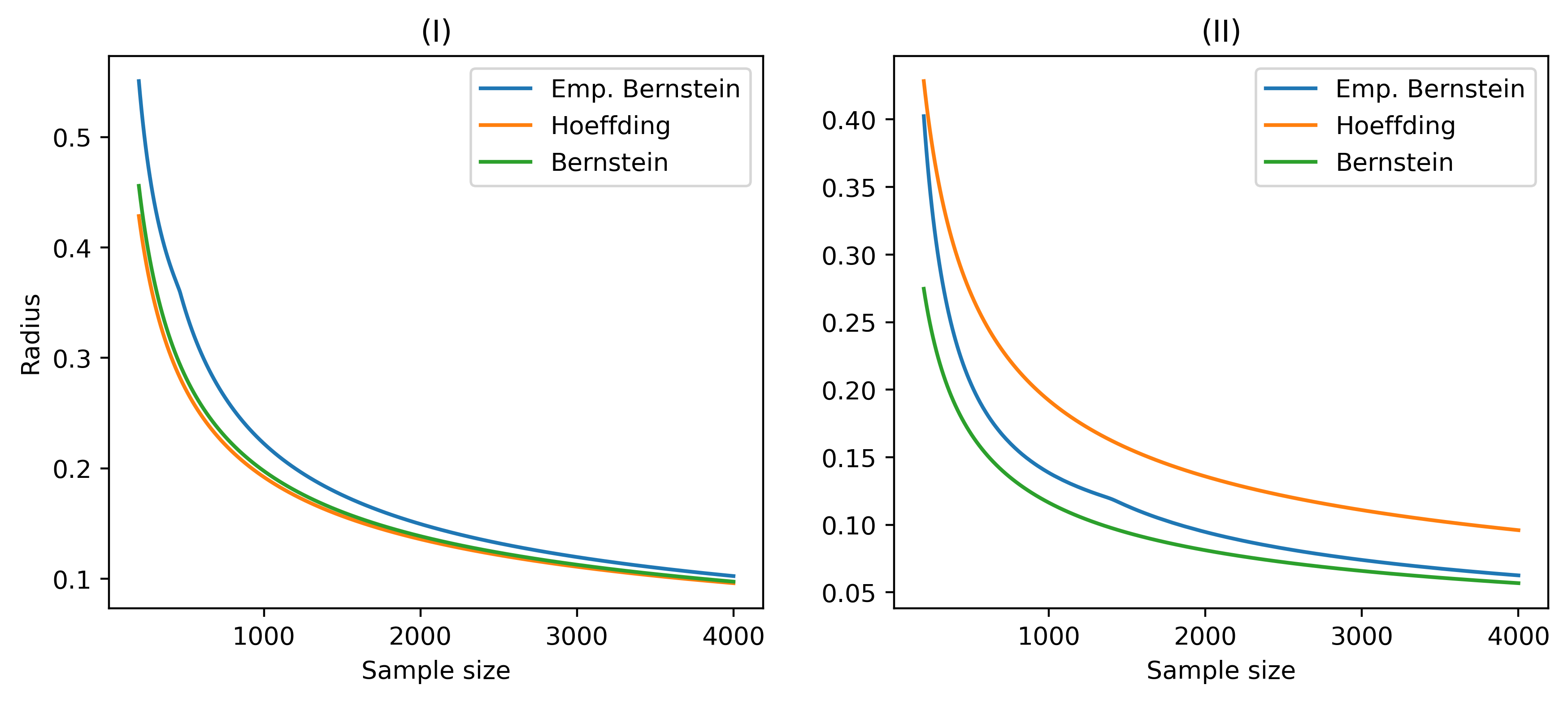}
  \caption{Radius of the confidence sets given by the empirical Bernstein (our proposal) and Hoeffding/Bernstein type  bounds \cite{pinelis1994optimum}, averaged over $100$ repetitions. The oracle Bernstein-type bound is the target to which we compare the other two. The Hoeffding-type bound is slightly better in the left plot (high variance) but much worse in the right plot (low variance). The empirical Bernstein bound is only slightly worse on both plots, with the difference in widths (even after being multiplied by $\sqrt n$) vanishing asymptotically, as indicated by Proposition~\ref{proposition:radius_ball_iid}.} 
  \label{fig:ci}
\end{figure}

\begin{figure}[] 
    \center \includegraphics[width=\textwidth]{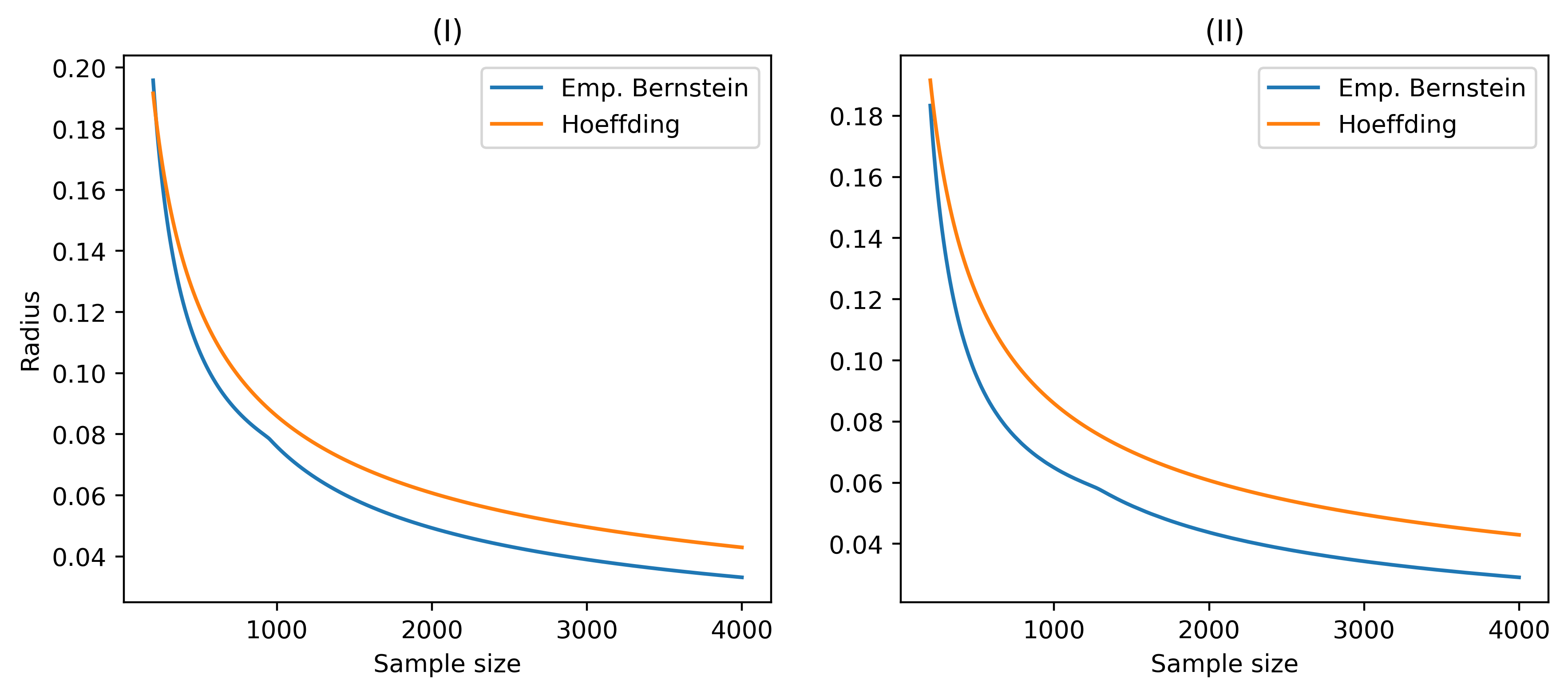}
  \caption{Radius of the RKHS-valued confidence sets given by the empirical Bernstein (our proposal) and Hoeffding-type  inequalities \cite{pinelis1994optimum}, averaged over $100$ repetitions. The empirical Bernstein inequality outperforms the Hoeffding-type inequality in both scenarios.} 
  \label{fig:ci_kernel}
\end{figure}

\subsubsection*{RKHS-valued data}  Figure~\ref{fig:ci_kernel} exhibits the empirical radii of the Hoeffding-type sets alongside our proposed empirical Bernstein confidence sets, for data in a reproducing kernel Hilbert space (RKHS). In both plots, the observations are the RBF kernel embeddings of variables drawn independently from  (I) a one-dimensional Rademacher distribution (left) or (II) a one-dimensional uniform distribution on $[-1, 1]$ (right). In both cases, the observations are bounded by $B = 1$. None of the variances of the embedded random variables is equal to the bound $1$, and so the empirical Bernstein inequality outperforms the Hoeffding-type inequality in either scenario.


\section{Proof of Theorem \ref{theorem:main_theorem}} \label{section:main_proof}

A simpler proof of Theorem~\ref{theorem:main_theorem} restricted to separable real Hilbert spaces may be found in Appendix~\ref{appendix:main_proof_HS}; the reader may find it helpful as it elucidates the same key ideas without dealing with the extra technical challenges that arise in Banach spaces. We defer the proofs of all the auxiliary lemmas below to Appendix~\ref{section:aux_lemmas}.

\begin{proof}[Proof of Theorem~\ref{theorem:main_theorem}]
    
The proof follows a similar high level strategy to \cite{pinelis1994optimum}, Theorem 3.2, with several additional technical challenges to deal with the estimated variance and the $\psi_{E}$ function defined above. 

First, we prove that it suffices to establish the result for random variables that only take a finite number of values in a finite dimensional Banach space. Second, we smoothen the norm in the finite dimensional Banach space so that it is Fréchet differentiable at any arbitrary point. Third, we derive two inequalities that play a central role in the proof. Fourth, we conclude the result by constructing a nonnegative supermartingale, while exploiting the Fréchet differentiability of the smoothened function in the process alongside the previous inequalities.

\subsubsection*{Step 1: From arbitrary Banach spaces to finite dimensional Banach spaces}

Our first goal is to prove that, if \eqref{eq:main_process} is a supermartingale for random variables that take a finite number of values in finite dimensional Banach spaces, then Theorem~\ref{theorem:main_theorem} follows. In order to do so, we approximate the original process by processes that only take a finite number of values. \changed{8}{The following lemma extends \cite{pinelis1994optimum}, Lemma 2.3, to martingale difference sequences while accommodating a predictable process.}

\begin{lemma} \label{lemma:pinelis2}
\begin{changedblock}{9}
    Let $(M_t)_{t \geq 1}$ be an integrable martingale difference sequence in a separable Banach space $(\X, \| \cdot \|)$ relative to a filtration $(\F_t)_{t \geq 0}$, so that $M_t$ is $\F_t$-measurable and $\E[M_t\mid\F_{t-1}]=0$. Let $(\lambda_t)_{t \geq 1}$ be an integrable predictable real-valued sequence relative to the same filtration. Then, for any $\epsilon > 0$, there exist a martingale difference sequence $(M_{t, \epsilon})_{t \geq 1}$ and a predictable sequence $(\lambda_{t, \epsilon})_{t \geq 1}$, both relative to a filtration $(\F_{t, \epsilon})_{t \geq 0}$, such that
    \begin{itemize}
        \item $M_{t, \epsilon}$ and $\lambda_{t+1, \epsilon}$ are random variables having only a finite number of values,
        \item $V_{t, \epsilon} \to V_t$ in probability as $\epsilon \downarrow 0$ for any $t \geq 1$,
    \end{itemize}
    where, with $M_0=M_{0,\epsilon}=0$,
    \begin{align*}
        V_t = ((0, \lambda_1), (M_1, \lambda_2), \ldots, (M_t, \lambda_{t+1})), \quad V_{t, \epsilon} = ((0, \lambda_{1, \epsilon}), (M_{1, \epsilon}, \lambda_{2, \epsilon}), \ldots, (M_{t, \epsilon}, \lambda_{t+1, \epsilon})).
    \end{align*}
\end{changedblock}
\end{lemma}
Let $V_t$ and $V_{t, \epsilon}$ be defined as in the lemma above, with $M_i=X_i-\mu$ and $M_0=0$. \changed{10}{Set
$X_{i,\epsilon}:=M_{i,\epsilon}+\mu=\E[X_i\mid\F_{i,\epsilon}]$. Then
$\E[X_{i,\epsilon}\mid\F_{i-1,\epsilon}]=\mu$ and, by conditional Jensen's inequality,
$\|X_{i,\epsilon}\|\leq\E[\|X_i\|\mid\F_{i,\epsilon}]\leq B$. Moreover, $(\lambda_{i,\epsilon})$ is predictable and remains in $(0,0.8]$. Thus the approximating observations and weights satisfy the assumptions of Theorem~\ref{theorem:main_theorem}.} Lemma~\ref{lemma:pinelis2} implies that $\ldba V_{t, \epsilon} - V_t \rdba_\infty \stackrel{}{\to} 0$ in probability as $\epsilon \downarrow 0$. Now note that the function $\Pi$ that maps $((m_0, \lambda_1), \ldots, (m_t, \lambda_{t+1}))$ to 
\begin{align*}
\lp\cosh \lp \ldba \sum_{i = 1}^j \lambda_i \frac{m_i}{4BD} \rdba \rp \exp \lp - \sum_{i = 1}^j \frac{\psi_E(\lambda_i)}{(4B)^2} \| m_i + \mu - \bar\mu_{i-1} \|^2  \rp \rp_{j=1}^t    
\end{align*}
is continuous at every possible argument, since all the weights lie in $(0,0.8]$. Note that $m_0$ does not appear on the right hand side above. \changed{11}{The relevant range of every weight is therefore contained in $[0,1]$,} and the tuples $V_t$ and $V_{t,\epsilon}$ take values in a finite product of $\X\times\R$, which is a separable Banach space. By the continuous mapping theorem on separable Banach spaces (see e.g. \cite{bosq2000linear}, Theorem 2.3), it follows that $\ldba \Pi(V_{t, \epsilon}) - \Pi(V_t) \rdba_\infty \stackrel{}{\to} 0$ in probability as $\epsilon \downarrow 0$. Furthermore, $\sup_\epsilon\ldba \Pi(V_{t, \epsilon}) \rdba_\infty \leq \cosh(\frac{t}{2D})$.
\begin{changedblock}{12}
Indeed, for either $(M_i^\star,\lambda_i^\star)=(M_i,\lambda_i)$ or $(M_{i,\epsilon},\lambda_{i,\epsilon})$, one has $\|M_i^\star\|\leq2B$ and $0\leq\lambda_i^\star\leq1$, and hence
\begin{align*}
    \ldba \sum_{i = 1}^t \lambda_i^\star \frac{M_i^\star}{4BD} \rdba \leq  \sum_{i = 1}^t \lambda_i^\star \frac{1}{2D} \leq \frac{t}{2D}.
\end{align*}
\end{changedblock}

\begin{lemma} \label{lemma:supermartingale_approximations}
    For any $\epsilon > 0$, let $(W_{t, \epsilon})_{t \geq 1}$ be a real valued supermartingale. If, for all $t \geq 1$,
    \begin{itemize}
        \item  $(W_{1, \epsilon}, \ldots, W_{t, \epsilon})$ converges to $(W_{1}, \ldots, W_{t})$ in probability as $\epsilon \downarrow 0$,
        \item $\sup_{\epsilon}\ldba (W_{1, \epsilon}, \ldots, W_{t, \epsilon}) \rdba_\infty  \leq C_t$ for some $C_t < \infty$,
    \end{itemize}
    then $(W_t)_{t \geq 1}$ is a supermartingale. 
\end{lemma}

In view of Lemma~\ref{lemma:supermartingale_approximations} (taking $(W_{1, \epsilon}, \ldots, W_{t, \epsilon}) = \Pi(V_{t, \epsilon})$ and $(W_{1}, \ldots, W_{t}) = \Pi(V_{t})$), the convergence in probability of $\Pi(V_{t, \epsilon})$ to $\Pi(V_t)$, alongside its boundedness, suffices to preserve the supermartingale property from $\Pi(V_{t, \epsilon})$ to $\Pi(V_t)$. 
Note that $\Pi(V_t) = S_t$, and \changed{13}{$\Pi(V_{t, \epsilon})$ is a process of the form $S_t$ built from finite-valued martingale differences and predictable weights in a finite dimensional Banach space.}

\subsubsection*{Step 2: Smoothening the norm of a finite dimensional 2-smooth Banach space}

Before we turn to prove the supermartingale condition of such processes, let us note that $(V_{t, \epsilon})_{t \leq n}$ takes only a finite number of values in a separable Banach space, and thus it suffices to work with the finite dimensional Banach subspace spanned by those vectors. It turns out that, in finite dimensional Banach spaces, it is easy to approximate the norm by a smooth function. This will allow us to take derivatives in these `smoothed' spaces, and then generalize the result to the original norm by passing to the limit. We start by presenting \cite{pinelis1994optimum}, Lemma 2.2, stated below for completeness.

\begin{lemma} \label{lemma:pinelis1}
    Let $(\X, \| \cdot \|)$ be $(2, D)$-smooth finite-dimensional Banach space. Define
    \begin{align*}
        \Upsilon_\epsilon(x) = \sqrt{\int_\X \| x - \epsilon y \|^2 \gamma(dy)},
    \end{align*}
    where $\gamma$ is a zero-mean Gaussian measure on $\X$ with support $\X$. Then for $\epsilon > 0$, $\Upsilon_\epsilon$ has the Fréchet derivatives of any order, and the directional derivatives in any direction $v \in \X$ satisfy the inequalities
    \begin{align} \label{eq:main_condition_on_the_norm}
        |\Upsilon_\epsilon'(x)(v)| \leq \|v\|, \quad (\Upsilon^2_\epsilon)''(x)(v, v) \leq 2 D^2 \|v\|^2
    \end{align}
    for all $x \in \X$. Besides, for each $x\in\X$, $\Upsilon_\epsilon(x) \to \|x\|$ as $\epsilon \downarrow 0$.
\end{lemma}

Note that the existence of the Gaussian measure $\gamma$ in Lemma \ref{lemma:pinelis1} is evident since the Banach space is finite dimensional.  Furthermore, the variance of such a Gaussian measure is irrelevant (as long as its support is the whole space $\mathcal{X}$); intuitively, the role of the parameter $\epsilon$ is precisely to make this variance go to zero as $\epsilon \downarrow 0$. We highlight that $\Upsilon_\epsilon$ is not actually a norm
(unless $\epsilon = 0$), but only a smoothed convolution of the norm.

While the norm of a Banach space may not be Fréchet differentiable at any arbitrary point, its Gateaux derivative in the direction of the point always exists. To see this, take any \changed{14}{$x \in \X$}, and note that $\| x + tx \| = (1+t)\|x\|$ for $t \geq -1$. As a function of $t$, $(1+t)\|x\|$ is clearly $C^\infty$, and its derivative at $0$ takes the value $\|x\|$. This behaviour still holds for $\Upsilon_\epsilon$, as elucidated in the following lemma.

\begin{lemma} \label{lemma:limit_upsilon}
    $\lim_{\epsilon \downarrow 0}\Upsilon_\epsilon'(x)(x) = \| x \|$.
\end{lemma}

By linearity, Lemma~\ref{lemma:limit_upsilon} implies that $\lim_{\epsilon \downarrow 0}\Upsilon_\epsilon'(x)(\iota x) = \iota \| x \|$ for all nonnegative $\iota$. Lastly, we show that $\Upsilon_\epsilon$ can be upper bounded in a closed interval for a range of $\epsilon$. This will later allow us to use the dominated convergence theorem.
\begin{lemma} \label{lemma:dct}
    For a $(2, D)$-smooth Banach space, it holds that
    \begin{align*}
        \Upsilon_\epsilon \lp x  + \beta y \rp \leq \sqrt{2\changed{15}{c(x, y)^2} + 2D^2C}, \quad \forall (x, y)  \in \X^2, \quad \forall (\beta,\epsilon) \in [0, 1]^2,
    \end{align*}
    where $c(x, y) = \max(\|x\|, \|x+y\|)$, and $C = \int_X \|y\|^2 \gamma(dy)$.
\end{lemma}

\subsubsection*{Step 3: Deriving two central deterministic inequalities}

As we shall see in the next step, understanding the interplay between several terms of a Taylor expansion is  of crucial importance to conclude the result. In order to do so, we require the two deterministic inequalities presented in this step. We start by presenting a sharp inequality concerning the $\pe$ function.

\begin{lemma} \label{lemma:main_ineq}
For all $\lambda \in (0, 0.8]$, $x\in(0, 1/2]$, and $D \in [1, \infty)$,
\begin{align*}
    \frac{\lp\lambda x + \pe(\lambda)x^2\rp^2}{\lp\frac{\lambda x}{D} - \pe(\lambda)x^2\rp^2}\lp e^{\frac{\lambda x}{D} - \pe(\lambda)x^2} -\lp \frac{\lambda x}{D} - \pe(\lambda)x^2 \rp -1 \rp \leq \psi_E(\lambda)x^2.
\end{align*}
\end{lemma}

 In stark contrast to previously used deterministic inequalities in related contributions (see e.g., \cite{fan_exponential_2015}, Proposition 4.1) that may be proven based on the $2$nd degree Taylor coefficient, the proof of Lemma~\ref{lemma:main_ineq} requires to study up to the $7$th degree Taylor coefficient. Lastly, the following lemma establishes an elemental yet powerful inequality concerning hyperbolic functions. 

\begin{lemma} \label{lemma:ineq_cosh_sinh}
For all $x, y \in \R$, 
\begin{align*}
    \cosh(y + x) - x \sinh(y+x) - \cosh(y) \leq 0.
\end{align*}
\end{lemma}

\subsubsection*{Step 4: Concluding the proof with a supermartingale construction}

Throughout, we assume that $(X_1, \ldots, X_n)$ only takes a finite number of values.
The process $S_t$ is nonnegative given that both the $\exp$ and $\cosh$ functions are nonnegative, so all we ought to prove is that such a process is a supermartingale, i.e., 
\begin{align*}
    \E_{t-1} \lb S_t \rb \leq S_{t-1}. 
\end{align*}

Based on Lemma~\ref{lemma:pinelis1},
\begin{align*}
    S_t  = \lim_{\epsilon \downarrow 0} \lb \cosh \lp \Upsilon_\epsilon \lp  \sum_{i = 1}^t \lambda_i \frac{X_i- \mu}{4BD} \rp \rp \exp \lp - \sum_{i = 1}^t \frac{\psi_E(\lambda_i)}{(4B)^2} \| X_i - \bar\mu_{i-1} \|^2  \rp \rb,
\end{align*}
and so
\begin{align*}
    \E_{t-1}  S_t  = \E_{t-1} \lim_{\epsilon \downarrow 0} \lb  \cosh \lp \Upsilon_\epsilon \lp  \sum_{i = 1}^t \lambda_i \frac{X_i- \mu}{4BD} \rp \rp \exp \lp - \sum_{i = 1}^t \frac{\psi_E(\lambda_i)}{(4B)^2} \| X_i - \bar\mu_{i-1} \|^2  \rp\rb.
\end{align*}

Let us denote
\begin{align*}
    Y_i = \frac{X_i - \mu}{4B}, \quad \delta_t = \frac{\bar\mu_t - \mu}{4B},  \quad Z_i = Y_i - \delta_{i-1},
\end{align*}
as well as
\begin{align*}
    M_t = \sum_{i = 1}^t \lambda_i \frac{X_i- \mu}{4BD}, \quad R_t = \exp \lp - \sum_{i = 1}^t \frac{\psi_E(\lambda_i)}{(4B)^2} \| X_i - \bar\mu_{i-1} \|^2  \rp , \quad A_t = M_{t-1} + \frac{\lambda_t}{D} \delta_{t-1}.
\end{align*}

We define \changemark{16}
\begin{align*}
    \varphi_{t,\epsilon}(\beta) &= R_{t-1} \cosh \Upsilon_\epsilon \lp M_{t-1} +\frac{\lambda_t}{D} \delta_{t-1} +  \beta \frac{\lambda_t}{D} \lp\frac{X_t - \mu}{4B} - \delta_{t-1}\rp \rp 
    \\&\quad\times\exp \lp -\beta\frac{\psi_E(\lambda_t)}{(4B)^2} \| X_t - \bar\mu_{t-1} \|^2\rp
    \\& = R_{t-1} \changednomark{\cosh} \Upsilon_\epsilon \lp M_{t-1}+\frac{\lambda_t}{D}\delta_{t-1} + \beta \frac{\lambda_t}{D} \lp Y_t - \delta_{t-1} \rp \rp \exp \lp-\beta\psi_E(\lambda_t) \| Y_t - \delta_{t-1} \|^2\rp
    \\& = R_{t-1} \cosh \Upsilon_\epsilon \lp A_t  + \beta \frac{\lambda_t}{D} Z_t \rp \exp \lp-\beta\psi_E(\lambda_t) \| Z_t \|^2\rp.
\end{align*}


We exploit its Taylor expansion
\begin{align*}
    \varphi_{t,\epsilon}(1) = \varphi_{t,\epsilon}(0) + \varphi_{t,\epsilon}'(0) + \int_0^1 (1-\beta)\varphi_{t,\epsilon}''(\beta)d\beta.
\end{align*}

Now we observe that 
\begin{align}
    \E_{t-1}\lb S_t\rb &= \E_{t-1}\lb\lim_{\epsilon \downarrow 0} \varphi_{t,\epsilon}(1)\rb\nonumber
    \\&\stackrel{(i)}{=} \lim_{\epsilon \downarrow 0} \E_{t-1}\lb\varphi_{t,\epsilon}(1)\rb\nonumber
    \\&= \lim_{\epsilon \downarrow 0}\E_{t-1} \lb \varphi_{t,\epsilon}(0) + \varphi_{t,\epsilon}'(0) + \int_0^1 (1-\beta)\varphi_{t,\epsilon}''(\beta)d\beta \rb\nonumber
    \\&\leq \limsup_{\epsilon \downarrow 0} \E_{t-1} \lb \varphi_{t,\epsilon}(0)\rb +\limsup_{\epsilon \downarrow 0} \E_{t-1} \lb\varphi_{t,\epsilon}'(0) \rb\label{eq:three_parts}
    \\&\quad+ \limsup_{\epsilon \downarrow 0}\E_{t-1} \lb\int_0^1 (1-\beta)\varphi_{t,\epsilon}''(\beta)d\beta\rb.  \nonumber
\end{align}
where (i) is obtained \changed{17}{because the approximating variables take finitely many values}, and so the expectation is a finite linear combination (and the limit transition preserves linear combinations). Now we analyze the three terms. First, note that 
$\varphi_{t,\epsilon}(0) = \E_{t-1} \lb\varphi_{t,\epsilon}(0)\rb = R_{t-1} \cosh \Upsilon_\epsilon \lp A_t \rp$. By Lemma~\ref{lemma:pinelis1}, 
\begin{align} \label{eq:first_term}
    \limsup_{\epsilon \downarrow 0}  \varphi_{t,\epsilon}(0)= \limsup_{\epsilon \downarrow 0} \lb R_{t-1} \cosh \Upsilon_\epsilon \lp A_t \rp \rb =  R_{t-1} \cosh \lim_{\epsilon \downarrow 0} \Upsilon_\epsilon \lp A_t \rp  = R_{t-1} \cosh \ldba A_t \rdba.
\end{align}

Second,
\begin{align*}
\frac{\varphi_{t,\epsilon}'(\beta)}{R_{t-1}} &= \frac{d}{d\beta} \lp  \cosh \Upsilon_\epsilon \lp A_t + \beta \frac{\lambda_t}{D} Z_t \rp \exp \lp-\beta\psi_E(\lambda_t) \| Z_t \|^2\rp \rp
\\&=  \exp \lp-\beta\psi_E(\lambda_t) \| Z_t \|^2\rp  \sinh \Upsilon_\epsilon \lp A_t + \beta \frac{\lambda_t}{D} Z_t \rp \Upsilon_\epsilon' \lp A_t  + \beta \frac{\lambda_t}{D} Z_t \rp \lp \frac{\lambda_t}{D} Z_t \rp  
\\&\quad -   \psi_E(\lambda_t) \| Z_t \|^2\cosh \Upsilon_\epsilon \lp A_t + \beta \frac{\lambda_t}{D} Z_t \rp  \exp \lp-\beta\psi_E(\lambda_t) \| Z_t \|^2\rp.
\end{align*}

Hence 
\begin{align*}
    \frac{\varphi'_{t, \epsilon}(0)}{R_{t-1}}&=    \sinh \Upsilon_\epsilon \lp A_t \rp \Upsilon_\epsilon' \lp A_t  \rp \lp \frac{\lambda_t}{D} Z_t \rp  -  \psi_E(\lambda_t) \| Z_t \|^2\cosh \Upsilon_\epsilon \lp A_t \rp,
\end{align*}
and so \changemark{18}
\begin{align*}
    \changednomark{\E_{t-1}} \lb \frac{\varphi'_{t, \epsilon}(0)}{R_{t-1}}\rb &=   \changednomark{\E_{t-1}} \lb  \sinh \Upsilon_\epsilon \lp A_t \rp \Upsilon_\epsilon' \lp A_t  \rp \lp \frac{\lambda_t}{D} Z_t \rp  \rb-\changednomark{\E_{t-1}} \lb   \psi_E(\lambda_t) \| Z_t \|^2\cosh \Upsilon_\epsilon \lp A_t \rp\rb
    \\&=     \sinh \Upsilon_\epsilon \lp A_t \rp \Upsilon_\epsilon' \lp A_t  \rp \lp \changednomark{\E_{t-1}} \lb\frac{\lambda_t}{D} Z_t\rb \rp  -\changednomark{\E_{t-1}} \lb   \psi_E(\lambda_t) \| Z_t \|^2\cosh \Upsilon_\epsilon \lp A_t \rp\rb
    \\&=     \sinh \Upsilon_\epsilon \lp A_t \rp \Upsilon_\epsilon' \lp A_t  \rp \lp \changednomark{\E_{t-1}} \lb\frac{\lambda_t}{D} (Y_t - \delta_{t-1})\rb \rp  
    \\&\quad-\changednomark{\E_{t-1}} \lb   \psi_E(\lambda_t) \| Z_t \|^2\cosh \Upsilon_\epsilon \lp A_t \rp\rb
    \\&\stackrel{(i)}{=}    -\sinh \Upsilon_\epsilon \lp A_t \rp \Upsilon_\epsilon' \lp A_t  \rp \lp \frac{\lambda_t}{D} \delta_{t-1} \rp  -\changednomark{\E_{t-1}} \lb   \psi_E(\lambda_t) \| Z_t \|^2\cosh \Upsilon_\epsilon \lp A_t \rp\rb,
\end{align*}
where (i) is obtained due to $\E_{t-1}[Y_t] = 0$. Note that $\delta_{t-1}$  is simply $M_{t-1}$ rescaled by a positive factor, and so they are both aligned to $A_t = M_{t-1} + \frac{\lambda_t}{D} \delta_{t-1}$. 
Based on Lemma~\ref{lemma:limit_upsilon}, this implies that
\begin{align*}
    \lim_{\epsilon \downarrow 0}\lb\sinh \Upsilon_\epsilon \lp A_t \rp \Upsilon_\epsilon' \lp A_t  \rp \lp \frac{\lambda_t}{D} \delta_{t-1} \rp\rb &= \lb \lim_{\epsilon \downarrow 0}\sinh \Upsilon_\epsilon \lp A_t \rp \rb  \lb\lim_{\epsilon \downarrow 0}\Upsilon_\epsilon' \lp A_t  \rp \lp \frac{\lambda_t}{D} \delta_{t-1} \rp\rb
    \\&= \ldba\frac{\lambda_t}{D}  \delta_{t-1}\rdba\sinh (\|A_t\|).
\end{align*}
Thus,
\begin{align}
     \limsup_{\epsilon \downarrow 0}\changednomark{\E_{t-1}} \lb \frac{\varphi'_{t, \epsilon}(0)}{R_{t-1}}\rb &= - \frac{\lambda_t}{D} \ldba  \delta_{t-1}\rdba\sinh (\|A_t\|) -\lim_{\epsilon \downarrow 0}\changednomark{\E_{t-1}} \lb   \psi_E(\lambda_t) \| Z_t \|^2\cosh \Upsilon_\epsilon \lp A_t \rp\rb \nonumber
     \\&\stackrel{(i)}{=} - \frac{\lambda_t}{D} \ldba  \delta_{t-1}\rdba\sinh (\|A_t\|) -\changednomark{\E_{t-1}} \lb \lim_{\epsilon \downarrow 0} \lb   \psi_E(\lambda_t) \| Z_t \|^2\cosh \Upsilon_\epsilon \lp A_t \rp\rb\rb \nonumber
    \\&= - \frac{\lambda_t}{D} \ldba  \delta_{t-1}\rdba\sinh (\|A_t\|) -\changednomark{\E_{t-1}} \lb   \psi_E(\lambda_t) \| Z_t \|^2\cosh \|A_t\|\rb,\label{eq:second_term}
\end{align}
where (i) is obtained given that the limit transition preserves linear combinations. 

Third,
\begin{align*}
\frac{\varphi_{t,\epsilon}''(\beta)}{R_{t-1}} \exp \lp\beta\psi_E(\lambda_t) \| Z_t \|^2\rp &=   \cosh \Upsilon_\epsilon \lp A_t  + \beta \frac{\lambda_t}{D} Z_t \rp \lb\Upsilon_\epsilon' \lp A_t  + \beta \frac{\lambda_t}{D} Z_t \rp \lp \frac{\lambda_t}{D} Z_t \rp\rb^2
\\&\quad + \sinh \Upsilon_\epsilon \lp A_t  + \beta \frac{\lambda_t}{D} Z_t \rp \Upsilon_\epsilon'' \lp A_t  + \beta \frac{\lambda_t}{D} Z_t \rp \lp \frac{\lambda_t}{D} Z_t \rp 
\\&\quad -  2\psi_E(\lambda_t) \| Z_t \|^2\sinh \Upsilon_\epsilon \lp A_t  + \beta \frac{\lambda_t}{D} Z_t \rp 
\\&\quad\quad\times\Upsilon_\epsilon' \lp A_t + \beta \frac{\lambda_t}{D} Z_t \rp \lp \frac{\lambda_t}{D} Z_t \rp
\\&\quad +  \lp\psi_E(\lambda_t) \| Z_t \|^2\rp^2 \cosh \Upsilon_\epsilon \lp A_t  + \beta \frac{\lambda_t}{D} Z_t \rp.
\end{align*}

We note that $\sinh u \leq u \cosh u$ if, and only if, $u \geq 0$. This implies that, when $u'' u > 0$, one has $(\cosh u)'' = u'^2 \cosh u + u'' \sinh u \leq (u'^2 + u''u)\cosh u = \frac{1}{2}(u^2)''\cosh u$ . If $u''u \leq 0$, then $(\cosh u)'' \leq u'^2 \cosh u$. Thus $(\cosh u)'' \leq  \max(\frac{1}{2}(u^2)'', u'^2) \cosh u$. Taking $u = \Upsilon_\epsilon\lp A_t + \beta \lambda_tZ_t/D\rp$, and in view of Lemma~\ref{lemma:pinelis1}, 
\begin{align*}
    \max(\frac{1}{2}(u^2)'', u'^2) \leq D^2 \ldba \frac{\lambda_t}{D} Z_t \rdba^2 = \ldba \lambda_t Z_t \rdba^2.
\end{align*}
Consequently, $(\cosh u)''$ is equal to
\begin{align*}
    &\cosh \Upsilon_\epsilon \lp A_t  + \beta \frac{\lambda_t}{D} Z_t \rp \lb\Upsilon_\epsilon' \lp A_t  + \beta \frac{\lambda_t}{D} Z_t \rp \lp \frac{\lambda_t}{D} Z_t \rp\rb^2 
    \\&+ \sinh \Upsilon_\epsilon \lp A_t  + \beta \frac{\lambda_t}{D} Z_t \rp \Upsilon_\epsilon'' \lp A_t  + \beta \frac{\lambda_t}{D} Z_t \rp \lp \frac{\lambda_t}{D} Z_t \rp
\end{align*}
and is upper bounded by $\cosh \Upsilon_\epsilon \lp A_t  + \beta \frac{\lambda_t}{D} Z_t \rp \ldba \lambda_t Z_t \rdba^2$. Furthermore,
\begin{align*}
    \lba -  2\psi_E(\lambda_t) \| Z_t \|^2\sinh \Upsilon_\epsilon \lp A_t  + \beta \frac{\lambda_t}{D} Z_t \rp \Upsilon_\epsilon' \lp A_t + \beta \frac{\lambda_t}{D} Z_t \rp \lp \frac{\lambda_t}{D} Z_t \rp \rba
\end{align*}
is upper bounded by $2\psi_E(\lambda_t) \| Z_t \|^2\cosh \Upsilon_\epsilon \lp A_t  + \beta \frac{\lambda_t}{D} Z_t \rp \ldba \lambda_t Z_t \rdba$ given Lemma~\ref{lemma:pinelis1}, $D \geq 1$, and $|\sinh u| \leq \cosh u$. Thus,
\begin{align*}
\frac{\varphi_{t,\epsilon}''(\beta)}{R_{t-1}} \exp \lp\beta\psi_E(\lambda_t) \| Z_t \|^2\rp &\leq   \cosh \Upsilon_\epsilon \lp A_t  + \beta \frac{\lambda_t}{D} Z_t \rp \ldba \lambda_t Z_t \rdba^2 
\\&\quad+2\psi_E(\lambda_t) \| Z_t \|^2\cosh \Upsilon_\epsilon \lp A_t  + \beta \frac{\lambda_t}{D} Z_t \rp \ldba \lambda_t Z_t \rdba
\\&\quad +  \lp\psi_E(\lambda_t) \| Z_t \|^2\rp^2 \cosh \Upsilon_\epsilon \lp A_t  + \beta \frac{\lambda_t}{D} Z_t \rp
\\&= \lp\psi_E(\lambda_t) \| Z_t \|^2 + \| \lambda_tZ_t \|\rp^2\cosh \Upsilon_\epsilon \lp A_t  + \beta \frac{\lambda_t}{D} Z_t \rp . 
\end{align*}

We note that
\begin{align*}
    \limsup_{\epsilon \downarrow 0}\E_{t-1} \lb\int_0^1 (1-\beta)\varphi_{t,\epsilon}''(\beta)d\beta\rb \leq \E_{t-1} \lb\limsup_{\epsilon \downarrow 0}\int_0^1 (1-\beta)\varphi_{t,\epsilon}''(\beta)d\beta\rb,
\end{align*}
and we observe that $\limsup_{\epsilon \downarrow 0} \int_0^1 (1 - \beta)\varphi_{t,\epsilon}''(\beta)d\beta$ is upper bounded by 
\begin{align} \label{eq:second_derivative_dct}
     H_t \limsup_{\epsilon \downarrow 0} \int_0^1 (1 - \beta)\cosh \Upsilon_\epsilon \lp A_t  + \beta \frac{\lambda_t}{D} Z_t \rp \exp \lp-\beta\psi_E(\lambda_t) \| Z_t \|^2\rp  d\beta,
\end{align}
where $H_t = R_{t-1}\lp\psi_E(\lambda_t) \| Z_t \|^2 + \|\lambda_t Z_t \|\rp^2$.

In view of Lemma~\ref{lemma:dct}, $(1 - \beta)\cosh \Upsilon_\epsilon \lp A_t  + \beta \frac{\lambda_t}{D} Z_t \rp \exp \lp-\beta\psi_E(\lambda_t) \| Z_t \|^2\rp$ is dominated for all $\beta \in [0,1]$ and $\epsilon \in [0,1]$.  The dominated convergence theorem implies that the $\limsup$ on the right hand side of \eqref{eq:second_derivative_dct} is actually a $\lim$ and is equal to
\begin{align}
    \nonumber&H_t  \int_0^1 (1 - \beta)\lb\lim_{\epsilon \downarrow 0}\cosh \Upsilon_\epsilon \lp A_t  + \beta \frac{\lambda_t}{D} Z_t \rp\rb \exp \lp-\beta\psi_E(\lambda_t) \| Z_t \|^2\rp  d\beta
    \\\nonumber=&H_t  \int_0^1 (1 - \beta)\cosh \ldba  A_t  + \beta \frac{\lambda_t}{D} Z_t \rdba \exp \lp-\beta\psi_E(\lambda_t) \| Z_t \|^2\rp  d\beta
    \\\nonumber\leq&H_t  \int_0^1 (1 - \beta)\cosh \ldba  A_t   \rdba \exp \lp \beta \ldba  \frac{\lambda_t}{D} Z_t\rdba - \beta\psi_E(\lambda_t) \| Z_t \|^2\rp  d\beta
    \\\nonumber\stackrel{(i)}{=} &H_t \cosh \ldba A_t  \rdba\frac{  \exp \lp \ldba \frac{\lambda_t}{D} Z_t \rdba  - \psi_E(\lambda_t) \| Z_t \|^2\rp - \lp \ldba \frac{\lambda_t}{D} Z_t \rdba  - \psi_E(\lambda_t) \| Z_t \|^2\rp - 1}{\lp \ldba \frac{\lambda_t}{D} Z_t \rdba  - \psi_E(\lambda_t) \| Z_t \|^2\rp^2}
    \\\nonumber= &R_{t-1} \cosh \ldba A_t  \rdba \lp \psi_E(\lambda_t) \| Z_t \|^2 +  \ldba \lambda_t Z_t \rdba  \rp^2 \times
    \\\nonumber&\quad\times\frac{  \exp \lp \ldba \frac{\lambda_t}{D} Z_t \rdba  - \psi_E(\lambda_t) \| Z_t \|^2\rp - \lp \ldba \frac{\lambda_t}{D} Z_t \rdba  - \psi_E(\lambda_t) \| Z_t \|^2\rp - 1}{\lp \ldba \frac{\lambda_t}{D} Z_t \rdba  - \psi_E(\lambda_t) \| Z_t \|^2\rp^2}
    \\&\stackrel{(ii)}{\leq} R_{t-1}\cosh \ldba A_t  \rdba \changed{19}{\psi_E(\lambda_t)} \| Z_t \|^2, \nonumber
\end{align}
where equality (i) is obtained from $\int_0^1 (1-\beta) e^{a\beta} d\beta = \frac{e^a - a - 1}{a^2}$, and inequality  (ii) follows from Lemma~\ref{lemma:main_ineq}. Thus, 
\begin{align}
    \limsup_{\epsilon \downarrow 0}\E_{t-1} \lb\int_0^1 (1-\beta)\varphi_{t,\epsilon}''(\beta)d\beta\rb \leq \E_{t-1} \lb R_{t-1}\cosh \ldba A_t  \rdba \changednomark{\psi_E(\lambda_t)} \| Z_t \|^2 \rb \label{eq:third_term}.
\end{align}

We observe that \eqref{eq:three_parts} is upper bounded by the sum of \eqref{eq:first_term}, \eqref{eq:second_term}, and \eqref{eq:third_term}, which add up to 
\begin{align*}
    R_{t-1} \lc \cosh \ldba A_t \rdba  - \frac{\lambda_t}{D} \ldba  \delta_{t-1}\rdba\sinh (\|A_t\|)\rc .
\end{align*}

Given that $\delta_{t-1}$  is simply $M_{t-1}$ re-scaled by a positive factor, it follows that $A_t = \ldba M_{t-1} + \frac{\lambda_t}{D} \delta_{t-1} \rdba = \ldba M_{t-1}\rdba+\ldba \frac{\lambda_t}{D}  \delta_{t-1} \rdba$. Thus, 
\begin{align*}
    -\ldba  \frac{\lambda_t}{D}  \delta_{t-1} \rdba\sinh \ldba A_t  \rdba &= -\ldba \lambda_t \delta_{t-1} \rdba\sinh \ldba M_{t-1}+ \frac{\lambda_t}{D}  \delta_{t-1}  \rdba
    \\&=  -\ldba  \frac{\lambda_t}{D}  \delta_{t-1} \rdba\sinh \lp \ldba M_{t-1}\rdba+\ldba  \frac{\lambda_t}{D} \delta_{t-1}  \rdba\rp    
    \\&\stackrel{(i)}{\leq}  \cosh \ldba M_{t-1}  \rdba -  \cosh \lp \ldba M_{t-1}\rdba+\ldba  \frac{\lambda_t}{D} \delta_{t-1}  \rdba\rp  
    \\&\stackrel{}{=}  \cosh \ldba M_{t-1}  \rdba -  \cosh \ldba M_{t-1}+ \frac{\lambda_t}{D}  \delta_{t-1}  \rdba
    \\&\stackrel{}{=}  \cosh \ldba M_{t-1}  \rdba -  \cosh \ldba A_t  \rdba,
\end{align*}
where (i) is obtained in view of Lemma~\ref{lemma:ineq_cosh_sinh}. We conclude the proof by noting that
\begin{align*}
      \E_{t-1}\lb S_t\rb &\leq R_{t-1} \lc \cosh \ldba A_t \rdba  - \frac{\lambda_t}{D} \ldba  \delta_{t-1}\rdba\sinh (\|A_t\|)\rc 
      \\&\leq R_{t-1} \lc \cosh \ldba A_t \rdba + \cosh \ldba M_{t-1}  \rdba -  \cosh \ldba A_t  \rdba\rc 
      \\&= R_{t-1}\cosh \ldba M_{t-1} \rdba 
      \\&= S_{t-1}.
\end{align*}
\end{proof}

\section{Conclusion} \label{section:summary}

\subsection{Summary}

Current concentration inequalities for bounded vector-valued random variables encompass extensions of the Hoeffding and Bernstein inequalities. Although Bernstein's inequality is generally tighter than Hoeffding's, it \changed{20}{requires knowing} an upper bound on the variance of the random variables (which is generally unknown). The primary contribution of this current work is to provide a vector-valued empirical Bernstein concentration bound, which replaces the true variance by its empirical estimator. The subsequent confidence sets are typically tighter than those provided by the vector-valued Hoeffding inequality, and are actionable in practice. The proposed empirical Bernstein holds in 2-smooth separable Banach spaces, which include finite dimensional Euclidean spaces and separable Hilbert spaces. The bound extends the empirical Bernstein concentration inequality from \cite{waudby2024estimating} to this more general setting, and its proof combines novel deterministic inequalities with some of the theoretical tools presented in \cite{pinelis1994optimum} to derive a nonnegative supermartingale construction. This nonnegative supermartingale, in conjunction with Ville's inequality, gives way to both fixed sample size confidence sets (batch setting) and anytime-valid confidence sequences (sequential setting). The empirical performance of the proposed empirical Bernstein for different multivariate bounded distributions is also exhibited. 

\subsection{Future work}

\changed{21}{The need to develop} a series of very technical results may lead to an ad-hoc appearance of the work, but we believe that our techniques and overall approach can strongly inspire new concentration results. The techniques developed here most likely have immediate implications for novel vector-valued inequalities. Tight vector-valued concentration inequalities are yet to be developed for \changed{22}{sampling without replacement}, Markov-chain dependence, and self-normalized processes, among others; we expect our techniques to be recycled in this space. Generally, we regard Pinelis' contributions as a cornerstone for vector-valued concentration inequalities. Nonetheless, closing the gap between this cornerstone and particular problems is far from trivial. In this sense, we believe that our work establishes a general approach and proof techniques that can be used in related problems that are also not direct implications of Pinelis' contributions. 
 
The approach developed may also deeply inspire research outside of vector-valued concentration.  In our efforts of extending concentration inequalities to vector-valued data, we had to establish  a number of results that were substantially tighter than previous efforts, as previous techniques were insufficient in our more complex setting (e.g., our Lemma~\ref{lemma:main_ineq} replaces the inequality of \cite[Proposition 4.1]{fan_exponential_2015}, because we needed a significantly tighter deterministic inequality). Our contribution elucidates how previous efforts can be highly refined to yield powerful results in non-scalar settings, thus setting a strong foundation for other non-scalar valued problems. For example, dimension-free matrix-valued (or more generally, operator-valued) concentration inequalities compose an exciting, open line of research, and we expect that novel operator-valued inequalities may follow a similar approach to the one established here.  

\subsection{Applications}

Mean estimation and inference in Hilbert spaces is of special interest in the fields of functional data analysis (FDA) and reproducing kernel Hilbert space (RKHS) methods, among others. In FDA, the random variables are usually assumed to belong to the space of integrable functions over some set, commonly the unit interval \cite{ramsay2005functional, horvath2012inference, wang2016functional}. While some of the FDA literature uses the sup-norm \cite{degras2017simultaneous, cai2019simultaneous}, there is a rich body of literature on $L^2$-confidence balls \cite{juditsky2003nonparametric, cai2006adaptive, bull2013adaptive, szabo2015honest}. Our vector-valued empirical-Bernstein inequality gives a natural, immediate route to finite-sample $L^2$-confidence balls. 

RKHS methods embed the original random variables in a Hilbert space (many times of infinite dimensions) by means of a reproducing kernel, giving way to methodology with significant impact in the statistics and machine learning communities \cite{scholkopf2002learning, gretton2005measuring, szekely2007measuring, gretton2012kernel}. RKHS methods have widely exploited concentration inequalities for developing procedures with finite-sample guarantees, see e.g. \cite[Theorem 7]{gretton2012kernel}, \cite[Theorem 1]{lopez2015towards}, \cite[Theorem 2]{schneider2016probability}, \cite[Proposition A.1]{tolstikhin2017minimax}, \cite[Theorem 4.1]{chatalic2022nystrom}, \cite[Proposition 4]{sun2023kernel}, or \cite[Theorem A.1]{wolfer2025variance}. Our empirical Bernstein inequality is expected to dominate all these previous efforts. In fact, our empirical Bernstein inequality has already been explored in \cite[Appendix D]{liu2024robustness}, outperforming all of them. In the RKHS-valued experiments presented in Section~\ref{section:experiments},  we see that our inequality is tighter than Pinelis' Hoeffding-type inequality, which is already a refinement of the inequality used in \cite[Appendix D]{liu2024robustness} that outperforms all its current competitors (except our own inequality).  

\appendix
\section{Proofs of auxiliary lemmas used in Theorem \ref{theorem:main_theorem}} \label{section:aux_lemmas}

\subsection{Proof of Lemma \ref{lemma:pinelis2}}

\begin{changedblock}{23}
Take $\X \times \R$ to be the direct product of Banach spaces with norm $\|(m,\gamma)\|=\max(\|m\|,|\gamma|)$ for all $(m,\gamma)\in\X\times\R$. Since $\X\times\R$ is separable, the law of each $(M_j,\lambda_{j+1})$ is tight. Set $M_0=0$. For every $j\geq0$, choose a compact set $K_{j,\epsilon}\subset\X\times\R$ such that
\begin{align*}
    \Pb\big((M_j,\lambda_{j+1})\in K_{j,\epsilon}\big)
    \geq 1-\epsilon 2^{-(j+1)}.
\end{align*}
Cover $K_{j,\epsilon}$ by finitely many open balls of radius $\epsilon$, intersect the balls with $K_{j,\epsilon}$, and disjointify the resulting cover of $K_{j,\epsilon}$. Together with $(\X\times\R)\setminus K_{j,\epsilon}$, the resulting sets form a finite Borel partition $\mathcal P_{j,\epsilon}$ of $\X\times\R$ whose cells contained in $K_{j,\epsilon}$ have diameter at most $2\epsilon$.

Let $\F_{t,\epsilon}$ be the $\sigma$-field generated by all events
\begin{align*}
    \{(M_i,\lambda_{i+1})\in A\},
    \qquad i=0,1,\ldots,t,\quad A\in\mathcal P_{i,\epsilon},
\end{align*}
and define
\begin{align*}
    M_{t,\epsilon}=\E[M_t\mid\F_{t,\epsilon}],
    \qquad
    \lambda_{t+1,\epsilon}=\E[\lambda_{t+1}\mid\F_{t,\epsilon}].
\end{align*}
The filtration $(\F_{t,\epsilon})_{t\geq0}$ is finite and satisfies $\F_{t,\epsilon}\subseteq\F_t$. Hence $M_{t,\epsilon}$ and $\lambda_{t+1,\epsilon}$ are finite-valued, while $(\lambda_{t,\epsilon})$ is predictable. Moreover,
\begin{align*}
\E[M_{t,\epsilon}\mid\F_{t-1,\epsilon}]
&=\E[M_t\mid\F_{t-1,\epsilon}]\\
&=\E\big[\E[M_t\mid\F_{t-1}]\mid\F_{t-1,\epsilon}\big]
=0,
\end{align*}
so $(M_{t,\epsilon})$ is a martingale difference sequence.

On the event $(M_i,\lambda_{i+1})\in K_{i,\epsilon}$, let $A_{i,\epsilon}(\omega)$ be the atom of $\F_{i,\epsilon}$ containing $\omega$. The values of $(M_i,\lambda_{i+1})$ throughout this atom lie in the same cell of $\mathcal P_{i,\epsilon}$ as $(M_i(\omega),\lambda_{i+1}(\omega))$, and that cell has diameter at most $2\epsilon$. The conditional expectation of the pair therefore belongs to the closed convex hull of that cell, and
\begin{align*}
    \big\|(M_{i,\epsilon},\lambda_{i+1,\epsilon})-(M_i,\lambda_{i+1})\big\|
    \leq 2\epsilon.
\end{align*}
By the union bound, this inequality holds simultaneously for all $i\leq t$ with probability at least
\begin{align*}
    1-\sum_{i=0}^t\epsilon 2^{-(i+1)}\geq1-\epsilon.
\end{align*}
Consequently, $V_{t,\epsilon}\to V_t$ in probability as $\epsilon\downarrow0$ for every fixed $t$.
\end{changedblock}

\subsection{Proof of Lemma~\ref{lemma:supermartingale_approximations}}

Let us start by presenting an additional lemma, proved in a later subsection. 

\begin{lemma} \label{lemma:conditional_to_f}
Let $(W_t)_{t \geq 1}$ be a real, integrable random process. The statements
\begin{longlist}
    \item $\E\lb W_t | \sigma \lp W_1, \ldots, W_{t-1}\rp \rb \leq 0$,
    \item $\E\lb W_t f(W_1, \ldots, W_{t-1}) \rb \leq 0$ for all nonnegative bounded measurable functions $f$,
    \item $\E\lb W_t f(W_1, \ldots, W_{t-1}) \rb \leq 0$ for all nonnegative bounded continuous functions $f$,
\end{longlist}
are equivalent. 
\end{lemma}

    Based on Lemma~\ref{lemma:conditional_to_f}, it suffices to check that $\E\lb \lp W_{t} - W_{t-1}\rp f\lp W_{1}, \ldots, W_{t-1}\rp\rb \leq 0$ for all nonnegative bounded continuous functions $f$. 
    
    Let $f$ be an arbitrary nonnegative bounded continuous function, and let $C_f > 0$ such that $\|f\|_\infty \leq C_f$. Denote
    \begin{align*}
        Z_t = \lp W_{t} - W_{t-1}\rp f\lp W_{1}, \ldots, W_{t-1}\rp, \quad Z_{t, \epsilon} = \lp W_{t, \epsilon} - W_{t-1, \epsilon}\rp f\lp W_{1, \epsilon}, \ldots, W_{t-1, \epsilon}\rp.
    \end{align*}
    Note that $\E\lb Z_{t, \epsilon}\rb \leq 0$ given that $(W_{t, \epsilon})_{t \geq 1}$ is a supermartingale. Further, $\|Z_{t, \epsilon}\|_\infty \leq C := (C_t+C_{t-1})C_f$.  For $\delta > 0$, it follows that
    \begin{align*}
        \E\lb Z_t\rb &= \E\lb Z_t \ind(Z_t - Z_{t, \epsilon} < \delta )\rb + \E\lb Z_t \ind(Z_t - Z_{t, \epsilon} \geq \delta )\rb
        \\&\leq \E\lb (Z_{t, \epsilon} + \delta) \ind(Z_t - Z_{t, \epsilon} < \delta )\rb + \E\lb |Z_t| \ind(Z_t - Z_{t, \epsilon} \geq \delta )\rb
         \\&= \E\lb Z_{t, \epsilon} + \delta\rb - \E\lb (Z_{t, \epsilon} + \delta) \ind(Z_t - Z_{t, \epsilon} \geq \delta )\rb + \E\lb |Z_t| \ind(Z_t - Z_{t, \epsilon} \geq \delta )\rb
         \\&\leq  \delta - (-C + \delta) \E\lb  \ind(Z_t - Z_{t, \epsilon} \geq \delta )\rb + \E\lb |Z_t| \ind(Z_t - Z_{t, \epsilon} \geq \delta )\rb
         \\&=  \delta + (C - \delta) \Pb\lb  Z_t - Z_{t, \epsilon} \geq \delta \rb + \E\lb |Z_t| \ind(Z_t - Z_{t, \epsilon} \geq \delta )\rb.
    \end{align*}
    From the convergence in probability of $(W_{1, \epsilon}, \ldots, W_{t, \epsilon})$ to $(W_{1}, \ldots, W_{t})$ and the continuous mapping application, $Z_{t, \epsilon}$ converges to $Z_t$ in probability as $\epsilon \downarrow 0$. Thus, $\Pb\lb  Z_t - Z_{t, \epsilon} \geq \delta \rb$ converges to zero as $\epsilon \downarrow 0$, and $|Z_t| \ind(Z_t - Z_{t, \epsilon} \geq \delta )$ converges to $0$ in probability as $\epsilon \downarrow 0$. 

    Consequently, there exists a sequence $(\epsilon_k)_{k\geq1}$ such that $\epsilon_k$ and $\E\lb |Z_t| \ind(Z_t - Z_{t, \epsilon_k} \geq \delta ) \rb$ converge to zero as $k \to \infty$.  This may be seen by applying the dominated convergence theorem with the dominating integrable function $|Z_t|$ to a sequence $\lp |Z_t| \ind(Z_t - Z_{t, \epsilon_k} \geq \delta )\rp_{k \geq 1}$ that converges almost surely to $0$ (and with $\epsilon_k$ converging to $0$), whose existence is derived from the convergence in probability of $|Z_t| \ind(Z_t - Z_{t, \epsilon} \geq \delta )$ to 0 as $\epsilon \downarrow 0$. It follows that
    \begin{align*}
         \E\lb Z_t\rb \leq\lim_{k \to \infty} \lc   \delta + (C - \delta) \Pb\lb  Z_t - Z_{t, \epsilon_k} \geq \delta \rb + \E\lb |Z_t| \ind(Z_t - Z_{t, \epsilon_k} \geq \delta )\rb\rc = \delta + 0 + 0 = \delta.
    \end{align*}
     Given that $\delta$ is arbitrary, we conclude that $ \E\lb Z_t\rb \leq 0$.

   It only remains to prove Lemma~\ref{lemma:conditional_to_f}, which we accomplish later. This completes the proof of Lemma~\ref{lemma:supermartingale_approximations}.

\subsection{Proof of Lemma \ref{lemma:limit_upsilon}}

Note that 
\begin{align*}
    \Upsilon_\epsilon'(x)(x) = \lim_{t \to 0} \frac{\sqrt{\int_\X \| (1+t)x - \epsilon y \|^2 \gamma(dy)} - \sqrt{\int_\X \| x - \epsilon y \|^2 \gamma(dy)}}{t} = \frac{\partial\Xi}{\partial t} (\epsilon, 0),
\end{align*}
    where $\Xi(\epsilon, t) = \sqrt{\int_\X \| (1+t)x - \epsilon y \|^2 \gamma(dy)}$.   By construction, $\Xi$ is smooth, and so 
    \begin{align*}
        \lim_{\epsilon \downarrow 0}\Upsilon_\epsilon'(x)(x) = \lim_{\epsilon \downarrow 0}\frac{\partial\Xi}{\partial t} (\epsilon, 0) = \frac{\partial\Xi}{\partial t} (0, 0) = \|x\|,
    \end{align*}
    the last equality holding because $\Xi(0, t) = \|(1+t)x\|$. 

\subsection{Proof of Lemma \ref{lemma:dct}}

By definition of a $(2, D)$-smooth Banach space, for all $0\leq\epsilon < 1$,
    \begin{align*}
        \Upsilon_\epsilon(z) &= \sqrt{\int_\X \| z - \epsilon y \|^2 \gamma(dy)}
        \\&\leq  \sqrt{\int_\X 2\| z \|^2 + 2D^2\| \epsilon y \|^2 \gamma(dy)}
        \\&\leq  \sqrt{ 2\| z \|^2 + 2D^2\int_\X\|  y \|^2 \gamma(dy)}.
    \end{align*}
    It suffices to note that $\|z\| \leq c(x, y)$ for all $z = x + \beta y$, $\beta \in [0, 1]$.

\subsection{Proof of Lemma \ref{lemma:main_ineq}}

Let us define 
\begin{align*}
    g(x, \lambda, D) := -\psi_E(\lambda)x^2 + \frac{\lp\lambda x + \pe(\lambda)x^2\rp^2}{\lp\frac{\lambda x}{D} - \pe(\lambda)x^2\rp^2}\lp e^{\frac{\lambda x}{D} - \pe(\lambda)x^2} -\lp \frac{\lambda x}{D} - \pe(\lambda)x^2 \rp -1 \rp
\end{align*}
for $\lambda \in (0, 0.8]$, $x\in(0, 1/2]$, and $D \in [1, \infty)$.

We prove the inequality in four steps:
\begin{longlist}
    \item First, we upper bound $g(x, \lambda, D)$ by a function $f(x, \lambda, D)$, so it suffices to show that $f \leq 0$.
    \item Second, we show that $f(x, \lambda, 1) \leq 0$ implies $f(x, \lambda, D) \leq 0$ for all $D \geq 1$.
    \item Third, we prove that $f(x, \lambda, 1)$ is nonpositive as long as a univariate function $h(\lambda)$ is also nonpositive. 
    \item Lastly, we show that $h(\lambda) \leq 0$ for all $\lambda \in [0, 0.8]$.
\end{longlist}

It remains to prove each of the steps. 

\textit{Details of step 1.}
Define
    \begin{align*}
        q(y) = 1 + y + \frac{y^2}{2} + \frac{y^3}{6} + \frac{y^4}{18}, \quad y \in \R.
    \end{align*}
    We observe that $e^y \leq q(y)$ for all $y \in [-1, 1]$. Indeed, 
    \begin{align*}
        q(y) - e^y &= \frac{y^4}{18} - \sum_{k=4}^\infty \frac{y^k}{k!}= y^4 \lp \frac{1}{18} - \sum_{k=4}^\infty \frac{y^{k-4}}{k!} \rp 
        \\&\geq y^4 \lp \frac{1}{18} - \sum_{k=4}^\infty \frac{1}{k!} \rp = y^4 \lp \frac{1}{18} - \lp e - \frac{8}{3} \rp\rp 
        \\&\geq 0, 
    \end{align*}
    given that $\frac{1}{18} + \frac{8}{3} \approx 2.7222 > 2.7183 \approx e$. Thus $ g(x, \lambda, D) \leq  f(x, \lambda, D)$ for all $(x, \lambda) \in [0,0.5]\times[0, 0.8]$, where
    \begin{align*}
        f(x, \lambda, D) =  \lp\lambda x + \pe(\lambda)x^2\rp^2\lp \frac{1}{2} + \frac{\frac{\lambda x}{D} - \pe(\lambda)x^2 }{6} + \frac{\lp \frac{\lambda x}{D} - \pe(\lambda)x^2 \rp^2}{18}\rp-\psi_E(\lambda)x^2.
    \end{align*}

\textit{Details of step 2.} Let us assume that $f(x, \lambda, 1) \leq 0$ for all $(x, \lambda) \in [0,0.5]\times[0, 0.8]$. We observe that  

\begin{align*}
    \frac{\partial f}{\partial D}(x, \lambda, D) =  - \frac{\lambda x}{D^2}\lp\lambda x + \pe(\lambda)x^2\rp^2\lp \frac{ 1}{6} + \frac{\frac{\lambda x}{D} - \pe(\lambda)x^2 }{9}\rp \leq 0,
\end{align*}
given that $\lba \frac{\lambda x}{D} - \pe(\lambda)x^2 \rba \leq 1$ implies
\begin{align*}
    \frac{ 1}{6} + \frac{\frac{\lambda x}{D} - \pe(\lambda)x^2 }{9} \geq 0.
\end{align*}

Hence, $f(x, \lambda, D)$ decreases with D. If $f(x, \lambda, 1)\leq0$, it follows that $f(x, \lambda, D) \leq 0$ for all $D \geq 1$.

\textit{Details of step 3.} It remains to prove that $f(x, \lambda, 1) \leq0$. Fix $\lambda \in (0, 0.8]$. We emphasize that $\pe$ is increasing and $\pe(\lambda) < 0.81$ on $(0, 0.8]$. Furthermore, $\lambda^2 \leq 2\pel$ for all $\lambda \in (0, 1]$. This may be seen noting that
\begin{align*}
    \pel = \sum_{k \geq 2} \frac{\lambda^k}{k} \geq \frac{\lambda^2}{2}, \quad \forall \lambda \geq 0.
\end{align*}
Moreover, $\pel \leq \frac{4}{3}\lambda^2$ for all $\lambda \in (0, 0.8]$. To see this, define the function $\eta(\lambda) = \pel-\frac{4}{3}\lambda^2$. Note that $\eta'(\lambda) = \frac{\lambda}{1 - \lambda} - \frac{8}{3}\lambda = \lambda (\frac{1}{1 - \lambda} - \frac{8}{3} )$ is nonpositive on $[0, 5/8]$ and nonnegative on $[5/8, 0.8]$. Thus, $\eta$ is nonincreasing on $\changed{24}{[0, 5/8]}$ and nondecreasing on $\changednomark{[5/8, 0.8]}$. It suffices to check that $\eta(0) = 0$ and $\eta(0.8) \approx -0.04 \leq 0$ to conclude that $\eta \leq 0$ on $[0, 0.8]$.

 We denote $\tilde f(x) = f(x, \lambda, 1)/ x^2$. Note that $\tilde f \leq 0$ if, and only if, $f \leq 0$. The function $\tilde f(x)$ is a polynomial of degree $6$. We observe that 
 \begin{align*}
        9\tilde f''(x) &=   \lambda^4 + 3\lambda^2\pel\lp 1-4\pel x^2\rp 
        \\&\quad+ 3\pe^2(\lambda) \lp -3 \lambda x + 5\pe^2(\lambda)x^4 - 6\pel x^2 + 3 \rp 
        \\&\stackrel{(i)}{\geq}  \lambda^4  + 3\pe^2(\lambda) \lp -3 \lambda x - 6\pel x^2 + 3 \rp 
        \\&\stackrel{}{=}   \lambda^4  + 9\pe^2(\lambda) \lp - \lambda x - 2\pel x^2 + 1 \rp 
        \\&\stackrel{(ii)}{\geq}  \lambda^4 
        \\&\geq 0.
    \end{align*}
    where (i) is obtained given that $ 1-4\pel x^2 \geq 0$ and $5\pe^2(\lambda)x^4 \geq 0$ for $(x, \lambda) \in [0, 0.5]\times[0, 0.8]$, and (ii) is obtained in view of $ 1- \lambda x - 2\pel x^2 \geq 0$ for $(x, \lambda) \in [0, 0.5]\times[0, 0.8]$. Thus, $\tilde f$ is convex on $[0, 1/2]$. Based on convexity and $\tilde f(0) = 0$, it remains to check that $\tilde f(0.5) \leq 0$ to conclude that $\tilde f \leq 0$ on $[0, 0.5]$. That is, the function
    \begin{align*}
        h(\lambda) := \frac{f(0.5, \lambda, 1)}{0.5^2} = -\psi_E(\lambda) + \lp\lambda  + \frac{\pel}{2}\rp^2\lp \frac{1}{2} + \frac{\frac{\lambda}{2} - \frac{\pe(\lambda)}{4}}{6} + \frac{\lp \frac{\lambda}{2} - \frac{\pe(\lambda)}{4} \rp^2}{18}\rp.
    \end{align*}
    fulfils $h(\lambda) \leq 0$ for all $\lambda \in \changed{25}{(0, 0.8]}$.

\textit{Details of step 4.}  We observe that
\begin{align}
    \nonumber12h(\lambda) &= -12\psi_E(\lambda) + \lp\lambda  + \frac{\pel}{2}\rp^2\lp 6 + \lambda - \frac{\pe(\lambda)}{2} + \frac{2\lp \frac{\lambda}{2} - \frac{\pe(\lambda)}{4} \rp^2}{3}\rp
    \\\nonumber&=  \lp\lambda^2  + \lambda \pel+ \frac{\pe^2(\lambda)}{4} \rp\lp 6 + \lambda - \frac{\pe(\lambda)}{2} +  \frac{\lambda^2}{6} + \frac{\pe^2(\lambda)}{24} - \frac{\lambda \pel}{6}\rp
    \\& \quad-12\psi_E(\lambda) \nonumber
    \\\nonumber&= -12\psi_E(\lambda) + 6\lambda^2 + \lambda^3 - \frac{\lambda^2\pe(\lambda)}{2} +  \frac{\lambda^4}{6} + \frac{\lambda^2\pe^2(\lambda)}{24} - \frac{\lambda^3 \pel}{6}
    \\\nonumber&\quad  + 6\lambda\pel + \lambda^2\pel - \frac{\lambda\pe^2(\lambda)}{2} +  \frac{\lambda^3\pel}{6} + \frac{\lambda\pe^3(\lambda)}{24} - \frac{\lambda^2 \pe^2(\lambda)}{6}
    \\\nonumber&\quad  + \frac{3}{2}\pe^2(\lambda) + \frac{\lambda\pe^2(\lambda)}{4} - \frac{\pe^3(\lambda)}{8} +  \frac{\lambda^2\pe^2(\lambda)}{24} + \frac{\pe^4(\lambda)}{96} - \frac{\lambda\pe^3(\lambda)}{24}
     \\\nonumber&= \underbrace{-12\psi_E(\lambda) + 6\lambda^2 + \lambda^3 + 6\lambda \pel + \frac{\lambda^2\pe(\lambda)}{2} +  \frac{\lambda^4}{6} + \frac{3}{2}\pe^2(\lambda)}_{\text{(I)}}  
     \\\nonumber&\quad\underbrace{- \frac{\lambda\pe^2(\lambda)}{4} - \frac{\lambda^2 \pe^2(\lambda)}{12} - \frac{\pe^3(\lambda)}{8} + \frac{\pe^4(\lambda)}{96}}_{\text{(II)}}. 
\end{align}

Note that $\frac{\pe^3(\lambda)}{8} > \frac{\pe^4(\lambda)}{96}$, and so (II) is negative. Thus it remains to prove that (I) is nonpositive. In view of $\pel = \sum_{k \geq2}\lambda^k/k$, it follows that 
\begin{align*}
    \text{(I)} &= -12 \sum_{k \geq 2}\frac{\lambda^k}{k} + 6\lambda^2 + \lambda^3 + 6 \sum_{k \geq 3}\frac{\lambda^k}{k-1} + \frac{1}{2} \sum_{k \geq 4}\frac{\lambda^k}{k-2} +  \frac{\lambda^4}{6} + \frac{3}{2}\lp\sum_{k \geq 2}\frac{\lambda^k}{k} \rp^2 
    \\&\stackrel{}{=} -12 \sum_{k \geq 2}\frac{\lambda^k}{k} + 6\lambda^2 + \lambda^3 + 6 \sum_{k \geq 3}\frac{\lambda^k}{k-1} + \frac{1}{2} \sum_{k \geq 4}\frac{\lambda^k}{k-2} +  \frac{\lambda^4}{6} + \frac{3}{2}\lp\frac{\lambda^2}{2} + \sum_{k \geq 3}\frac{\lambda^k}{k} \rp^2 
    \\&\stackrel{}{=} -12 \sum_{k \geq 2}\frac{\lambda^k}{k} + 6\lambda^2 + \lambda^3 + 6 \sum_{k \geq 3}\frac{\lambda^k}{k-1} + \frac{1}{2} \sum_{k \geq 4}\frac{\lambda^k}{k-2} +  \frac{\lambda^4}{6} 
    \\&\quad+ \frac{3}{2}\lb\frac{\lambda^4}{4} + \lambda^2\sum_{k \geq 3}\frac{\lambda^k}{k} + \lp \sum_{k \geq 3}\frac{\lambda^k}{k} \rp^2\rb
    \\&\stackrel{(i)}{\leq} -12 \sum_{k \geq 2}\frac{\lambda^k}{k} + 6\lambda^2 + \lambda^3 + 6 \sum_{k \geq 3}\frac{\lambda^k}{k-1} + \frac{1}{2} \sum_{k \geq 4}\frac{\lambda^k}{k-2} +  \frac{\lambda^4}{6} 
    \\&\quad+ \frac{3}{2}\lb\frac{\lambda^4}{4} + \lambda^2\sum_{k \geq 3}\frac{\lambda^k}{k} + \frac{5}{6}\lambda^2\sum_{k \geq 3}\frac{\lambda^k}{k}\rb 
    \\&\stackrel{}{=} -12 \sum_{k \geq 2}\frac{\lambda^k}{k} + 6\lambda^2 + \lambda^3 + 6 \sum_{k \geq 3}\frac{\lambda^k}{k-1} + \frac{1}{2} \sum_{k \geq 4}\frac{\lambda^k}{k-2} +  \frac{\lambda^4}{6} 
    \\&\quad+ \frac{3}{2}\lb\frac{\lambda^4}{4} + \sum_{k \geq 5}\frac{\lambda^k}{k-2} + \frac{5}{6}\sum_{k \geq 5}\frac{\lambda^k}{k-2}\rb 
    \\&\stackrel{}{=} -12 \sum_{k \geq 2}\frac{\lambda^k}{k} + 6\lambda^2 + 4\lambda^3 + \lambda^4\lp 2+ \frac{1}{4}+ \frac{1}{6}  + \frac{3}{8}\rp + 6\sum_{k \geq 5}\frac{\lambda^k}{k-1} + \frac{13}{4}\sum_{k \geq 5}\frac{\lambda^k}{k-2}
    \\&\stackrel{}{=} -12 \sum_{k \geq 2}\frac{\lambda^k}{k} + 6\lambda^2 + 4\lambda^3 + \lambda^4\lp 3-\frac{5}{24}\rp + 6\sum_{k \geq 5}\frac{\lambda^k}{k-1} + \frac{13}{4}\sum_{k \geq 5}\frac{\lambda^k}{k-2}
    \\&\stackrel{}{=} -\frac{5}{24}\lambda^4 -12 \sum_{k \geq 5}\frac{\lambda^k}{k} + 6\sum_{k \geq 5}\frac{\lambda^k}{k-1} + \frac{13}{4}\sum_{k \geq 5}\frac{\lambda^k}{k-2}
    \\&\stackrel{(ii)}{\leq} -\frac{5}{24}\lambda^5 -12 \sum_{k \geq 5}\frac{\lambda^k}{k} + 6\sum_{k \geq 5}\frac{\lambda^k}{k-1} + \frac{13}{4}\sum_{k \geq 5}\frac{\lambda^k}{k-2}
    \\&\stackrel{}{=} \lambda^5\lp  -\frac{5}{24} - \frac{12}{5} + \frac{6}{4} + \frac{13}{12} \rp -12 \sum_{k \geq 6}\frac{\lambda^k}{k} + 6\sum_{k \geq 6}\frac{\lambda^k}{k-1} + \frac{13}{4}\sum_{k \geq 6}\frac{\lambda^k}{k-2}
    \\&\stackrel{}{=} -\frac{1}{40}\lambda^5 -12 \sum_{k \geq 6}\frac{\lambda^k}{k} + 6\sum_{k \geq 6}\frac{\lambda^k}{k-1} + \frac{13}{4}\sum_{k \geq 6}\frac{\lambda^k}{k-2}
    \\&\stackrel{(iii)}{\leq}  -\frac{1}{40}\lambda^6 -12 \sum_{k \geq 6}\frac{\lambda^k}{k} + 6\sum_{k \geq 6}\frac{\lambda^k}{k-1} + \frac{13}{4}\sum_{k \geq 6}\frac{\lambda^k}{k-2}
    \\&\stackrel{}{=} \lambda^6 \lp -\frac{1}{40} - \frac{12}{6} + \frac{6}{5} +\frac{13}{16} \rp + \sum_{k \geq 7}\frac{\lambda^k}{k}\lp -12 + 6\frac{k}{k-1} +\frac{13}{4}\frac{k}{k-2} \rp
    \\&\stackrel{}{=} -\frac{1}{80}\lambda^6 + \sum_{k \geq 7}\frac{\lambda^k}{k}\lp -12 + 6\frac{k}{k-1} +\frac{13}{4}\frac{k}{k-2} \rp
    \\&\stackrel{}{\leq} \sum_{k \geq 7}\frac{\lambda^k}{k}\lp -12 + 6\frac{k}{k-1} +\frac{13}{4}\frac{k}{k-2} \rp,
\end{align*}
where (i) is obtained based on $\sum_{k \geq 3}\frac{\lambda^k}{k} = \pel - \frac{\lambda^2}{2} \leq \frac{4}{3}\lambda^2-\frac{\lambda^2}{2}= \frac{5}{6}\lambda^2$ for all $\lambda \in [0, 0.8]$, and (ii) and (iii) are obtained given that $\lambda \geq \lambda^2$ for all $\lambda \in [0, 1]$. Now define 
\begin{align*}
    \changed{26}{r\colon(2,\infty)\to\R,\qquad r(k)=-12+6\frac{k}{k-1}+\frac{13}{4}\frac{k}{k-2}.}
\end{align*}
The function $r$ is decreasing, and so it suffices to note that $r(7) = -0.45$ to conclude that $r(k) \leq 0$ for all $k \geq 7$, and thus 
\begin{align*}
     \sum_{k \geq 7}\changed{27}{\frac{\lambda^k}{k}}\lp -12 + 6\frac{k}{k-1} +\frac{13}{4}\frac{k}{k-2} \rp \leq 0.
\end{align*}

\subsection{Proof of Lemma \ref{lemma:ineq_cosh_sinh}}
    Define $f(x, y) = \cosh(y + x) - x \sinh(y+x) - \cosh(y)$. It follows that, for every $y$,
    \begin{align*}
        \frac{d}{dx}f(x, y) = - x \cosh(y + x)
    \end{align*}
    has the opposite sign of $x$, so the function decreases for $x \geq 0$ and increases for $x \leq 0$. Hence, its maximum is achieved at $0$. Further, $f(0, y) = 0$, and we conclude the result.

  \subsection{Proof of Lemma~\ref{lemma:conditional_to_f}}

    We prove $(iii)$ implies $(ii)$, $(ii)$ implies $(i)$, and $(i)$ implies $(iii)$, in order.
    
    $(iii) \implies (ii)$: Let $f$ be a nonnegative bounded measurable function. Given that $\R^{t-1}$ is locally compact, and both $\R^{t-1}$  and $\R$ are separable Banach spaces, Lusin's theorem \cite{lusin1912, bauldry2009introduction} yields that
        for any $\delta > 0$, there exists a continuous function $f_\delta$ such that $f_\delta = f$ on $A$ with $\Pb(A) \geq 1 - \delta$, and $\|f_\delta \|_\infty \leq \|f\|_\infty$.
    Without loss of generality, $f_\delta$ is nonnegative (otherwise, set the negative values to $0$).  By a small abuse of notation, we denote $f = f(W_1, \ldots, W_{t-1})$ and $f_\delta = f_\delta(W_1, \ldots, W_{t-1})$. For any $\delta > 0$, 
    \begin{align*}
        \E\lb W_t f \rb &= \E\lb W_t f \ind(f = f_\delta)\rb +  \E\lb W_t f \ind(f \neq f_\delta)\rb 
        \\&= \E\lb W_t f_\delta \ind(f = f_\delta)\rb +  \E\lb W_t f \ind(f \neq f_\delta)\rb  
        \\&= \E\lb W_t f_\delta \rb - \E\lb W_t f_\delta \ind(f \neq f_\delta)\rb +  \E\lb W_t f \ind(f \neq f_\delta)\rb
        \\&\leq - \E\lb W_t f_\delta \ind(f \neq f_\delta)\rb +  \E\lb W_t f \ind(f \neq f_\delta)\rb,
    \end{align*}
    where the inequality is obtained given that $f_\delta$ is a nonnegative continuous bounded function. We observe that
    \begin{align*}
        \lba- \E\lb W_t f_\delta \ind(f \neq f_\delta)\rb +  \E\lb W_t f \ind(f \neq f_\delta)\rb \rba &\leq \lp \|f_\delta\|_\infty + \|f\|_\infty\rp \E\lb |W_t| \ind(f \neq f_\delta) \rb
        \\&\leq 2 \|f\|_\infty \E\lb |W_t| \ind(f \neq f_\delta) \rb.
    \end{align*}
    Furthermore, there exists a sequence $(\delta_k)_{k\geq1}$ such that $\E\lb |W_t| \ind(f \neq f_{\delta_k}) \rb$ converges to zero as $k \to \infty$.  This may be seen by applying the dominated convergence theorem with the dominating integrable function $|W_t|$ to a sequence $( |W_t| \ind(f \neq f_{\delta_k}))_{k \geq 1}$ that converges almost surely to 0, whose existence is derived from the convergence in probability of $|W_t| \ind(f \neq f_\delta)$ to 0 as $\delta \downarrow 0$. We conclude that
    \begin{align*}
        \E\lb W_t f \rb &\leq  \lim_{k \to \infty}   \lc 2 \|f\|_\infty \E\lb |W_t| \ind(f \neq f_{\delta_k}) \rb \rc = 0.
    \end{align*}

    $(ii) \implies (i)$: Denote $\F_{t-1} = \sigma(W_1, \ldots, W_{t-1})$. Take 
    \begin{align*}
        f(W_1, \ldots, W_{t-1}) = \ind\lp \E\lb W_t | \F_{t-1}\rb   > 0 \rp.
    \end{align*}
    Note that 
    \begin{align*}
        \E \lb W_t f(W_1, \ldots, W_{t-1}) \rb &=  \E \lb W_t \ind\lp \E\lb W_t | \F_{t-1}\rb   > 0 \rp  \rb
        \\&=  \E\lb\E \lb W_t \ind\lp \E\lb W_t | \F_{t-1}\rb   > 0 \rp   | \F_{t-1}\rb\rb
        \\&=  \E\lb \ind\lp \E\lb W_t | \F_{t-1} \rb  > 0 \rp  \E \lb W_t   | \F_{t-1}\rb\rb
    \end{align*}
    is nonpositive by (ii). Given that $\ind\lp \E\lb W_t | \F_{t-1}\rb   > 0 \rp  \E \lb W_t   | \F_{t-1}\rb$ is nonnegative, and its expectation is nonpositive, it has to be zero almost surely. That is, $\E\lb W_t |\F_{t-1} \rb \leq 0$ almost surely.

    $(i) \implies (iii)$: From the tower property, 
    \begin{align*}
        \E\lb W_t f(W_1, \ldots, W_{t-1}) \rb &= \E\lb\E\lb W_t f(W_1, \ldots, W_{t-1}) | \sigma(W_1, \ldots, W_{t-1})\rb\rb
        \\&= \E\lb f(W_1, \ldots, W_{t-1}) \E\lb W_t  | \sigma(W_1, \ldots, W_{t-1})\rb\rb
        \\&\leq 0,
    \end{align*}
    where the inequality is obtained given that $\E\lb W_t  | \sigma(W_1, \ldots, W_{t-1})\rb \leq 0$ and $f$ is nonnegative.

\section{Alternative proof of Theorem \ref{theorem:main_theorem} restricted to separable real Hilbert spaces} \label{appendix:main_proof_HS}

The process \eqref{eq:main_process} is nonnegative given that both the $\exp$ and $\cosh$ functions are nonnegative, so all we have to prove is that such a process is a supermartingale. Let us denote
\begin{align*}
    Y_i = \frac{X_i - \mu}{4B}, \quad \delta_t = \frac{\bar\mu_t - \mu}{4B},  \quad Z_i = Y_i - \delta_{i-1},
\end{align*}
as well as
\begin{align*}
    M_t = \sum_{i = 1}^t \lambda_i \frac{X_i- \mu}{4B}, \quad R_t = \exp \lp - \sum_{i = 1}^t \frac{\psi_E(\lambda_i)}{(4B)^2} \| X_i - \bar\mu_{i-1} \|^2  \rp , \quad A_t = M_{t-1} + \lambda_t \delta_{t-1}.
\end{align*}

We define
\begin{align*}
    \varphi_t(\beta) &= R_{t-1} \cosh \ldba M_{t-1} +{\lambda_t} \delta_{t-1} +  \beta {\lambda_t} \lp\frac{X_t - \mu}{4B} - \delta_{t-1}\rp \rdba
    \\&\quad \times\exp \lp -\beta\frac{\psi_E(\lambda_t)}{(4B)^2} \| X_t - \bar\mu_{t-1} \|^2\rp
    \\& = R_{t-1} \cosh \ldba M_{t-1}+{\lambda_t}\delta_{t-1} + \beta {\lambda_t} \lp Y_t - \delta_{t-1} \rp \rdba \exp \lp-\beta\psi_E(\lambda_t) \| Y_t - \delta_{t-1} \|^2\rp
    \\& = R_{t-1} \cosh \ldba A_{t} + \beta {\lambda_t} Z_t \rdba \exp \lp-\beta\psi_E(\lambda_t) \| Z_t \|^2\rp.
\end{align*}
Note that, in order to conclude the result, it suffices to show that 
\begin{equation}\label{eq:suffice-to-show}
\E_{t-1}[\varphi_t(1)] \leq R_{t-1} \cosh  \| M_{t-1} \|.
\end{equation}

Assume for now that 
\begin{equation}\label{eq:simplfying-assumption}
A_{t} + \beta {\lambda_t} Z_t \neq 0 \text{ almost surely for all }\beta \in [0, 1].
\end{equation}
Since $\varphi_t$ is smooth on $[0,1]$, Taylor's theorem yields
\begin{align}\label{eq:taylor-hilbert}
    \varphi_t(1) = \varphi_t(0) + \varphi_t'(0) + \int_0^1 (1-\beta)\varphi_t''(\beta)d\beta.
\end{align}
We now analyze the three terms separately. First,   $\varphi_t(0) = R_{t-1}\cosh \ldba A_t \rdba$. For the second term, note that\footnote{For all $x \in H \backslash \{ 0 \}$, the Fréchet derivative and the second order Fréchet derivative of $\| \cdot \|$ at $x$ are denoted by $\|x \|'$ and $\|x \|''$ respectively. These Fréchet derivatives are operators fulfilling
\begin{align*}
    \| x \|'(\Delta) = \la \frac{x}{\|x\|}, \Delta\ra, \quad \| x \|''(\Delta, \Delta) = \frac{1}{\|x\|} \lp \|\Delta\|^2 - \la \Delta, \frac{x}{\|x\|}\ra^2 \rp,
\end{align*}
for all $x \in H \backslash \{ 0 \}$ and all $ \Delta \in H$.}
\begin{align*}
\frac{\varphi_t'(\beta)}{R_{t-1}}&= \frac{d}{d\beta} \lp  \cosh \ldba A_t + \beta {\lambda_t} Z_t \rdba \exp \lp-\beta\psi_E(\lambda_t) \| Z_t \|^2\rp \rp
\\&=  \exp \lp-\beta\psi_E(\lambda_t) \| Z_t \|^2\rp  \sinh \ldba A_t + \beta \lambda_t Z_t \rdba \la \frac{A_t  + \beta \lambda_t Z_t}{\left\| A_t  + \beta \lambda_t Z_t \right\|}, \lambda_t Z_t\ra  
\\&\quad \quad \quad -   \psi_E(\lambda_t) \| Z_t \|^2\cosh \ldba A_t + \beta \lambda_t Z_t \rdba  \exp \lp-\beta\psi_E(\lambda_t) \| Z_t \|^2\rp.
\end{align*}
Hence, the second term in~\eqref{eq:taylor-hilbert} evaluates to
\begin{align*}
    \frac{\varphi_t'(0)}{R_{t-1}}&=  \sinh \ldba A_t  \rdba \la \frac{{A_t}}{\left\| A_t  \right\|}, \lambda_t Z_t\ra  - \psi_E(\lambda_t) \| Z_t \|^2\cosh \ldba A_t \rdba .
\end{align*}
For the third term in~\eqref{eq:taylor-hilbert}, we calculate
\begin{align*}
\frac{\varphi_t''(\beta)}{R_{t-1}} \exp \lp\beta\psi_E(\lambda_t) \| Z_t \|^2\rp &=   \cosh \ldba A_t  + \beta \lambda_t Z_t \rdba \la \frac{A_t  + \beta \lambda_t Z_t}{\left\| A_t  + \beta \lambda_t Z_t \right\|}, \lambda_t Z_t\ra^2 
\\&\quad + \frac{\sinh \ldba A_t  + \beta \lambda_t Z_t \rdba}{\left\| A_t  + \beta \lambda_t Z_t \right\|} \la  \lambda_t Z_t, \lambda_t Z_t\ra 
\\&\quad - \frac{\sinh \ldba A_t  + \beta \lambda_t Z_t \rdba}{\left\| A_t  + \beta \lambda_t Z_t \right\|} \la \frac{A_t  + \beta \lambda_t Z_t}{\left\| A_t  + \beta \lambda_t Z_t \right\|}, \lambda_t Z_t\ra^2 
\\&\quad -  2\psi_E(\lambda_t) \| Z_t \|^2\sinh \ldba A_t  + \beta \lambda_t Z_t \rdba 
\\&\quad\quad\times\la \frac{A_t  + \beta \lambda_t Z_t}{\left\| A_t  + \beta \lambda_t Z_t \right\|}, \lambda_t Z_t\ra
\\&\quad +  \lp\psi_E(\lambda_t) \| Z_t \|^2\rp^2 \cosh \ldba A_t  + \beta \lambda_t Z_t \rdba.
\end{align*}
For $u \geq 0$, it holds that $\sinh u \leq u \cosh u$. Given that $ \la  \lambda_t Z_t, \lambda_t Z_t\ra - \la \frac{A_t  + \beta \lambda_t Z_t}{\left\| A_t  + \beta \lambda_t Z_t \right\|}, \lambda_t Z_t\ra^2 \geq 0$, the first three terms on the right hand side can be collapsed to yield
\begin{align*}
    \frac{\varphi_t''(\beta)}{R_{t-1}} \exp \lp\beta\psi_E(\lambda_t) \| Z_t \|^2\rp &\leq \cosh \ldba A_t  + \beta \lambda_t Z_t \rdba \la  \lambda_t Z_t, \lambda_t Z_t\ra 
\\&\quad -  2\psi_E(\lambda_t) \| Z_t \|^2\sinh \ldba A_t  + \beta \lambda_t Z_t \rdba 
\\&\quad\quad\times\la \frac{A_t  + \beta \lambda_t Z_t}{\left\| A_t  + \beta \lambda_t Z_t \right\|}, \lambda_t Z_t\ra
\\&\quad +  \lp\psi_E(\lambda_t) \| Z_t \|^2\rp^2 \cosh \ldba A_t  + \beta \lambda_t Z_t \rdba.
\end{align*}
Since $|\sinh u| \leq \cosh u$ for all $u \in \R$, and $\lba \la \frac{A_t  + \beta \lambda_t Z_t}{\left\| A_t  + \beta \lambda_t Z_t \right\|}, \lambda_t Z_t\ra\rba \leq \ldba \lambda_t Z_t \rdba$ for all $\beta \in [0,1]$, the second term further simplifies to yield
\begin{align*}
    \frac{\varphi_t''(\beta)}{R_{t-1}} \exp \lp\beta\psi_E(\lambda_t) \| Z_t \|^2\rp &\leq \cosh \ldba A_t  + \beta \lambda_t Z_t \rdba \la  \lambda_t Z_t, \lambda_t Z_t\ra 
\\&\quad +  2\psi_E(\lambda_t) \| Z_t \|^2\ldba \lambda_t Z_t \rdba \cosh \ldba A_t  + \beta \lambda_t Z_t \rdba  
\\&\quad +  \lp\psi_E(\lambda_t) \| Z_t \|^2\rp^2 \cosh \ldba A_t  + \beta \lambda_t Z_t \rdba
\\&= \cosh \ldba A_t  + \beta \lambda_t Z_t \rdba \lp \psi_E(\lambda_t) \| Z_t \|^2 +  \ldba \lambda_t Z_t \rdba  \rp^2
\\&\leq \cosh \ldba A_t  \rdba \exp \lp \ldba \beta \lambda_t Z_t \rdba \rp\lp \psi_E(\lambda_t) \| Z_t \|^2 +  \ldba \lambda_t Z_t \rdba  \rp^2.
\end{align*}
Since $\int_0^1 (1-\beta) e^{a\beta} d\beta = \frac{e^a - a - 1}{a^2}$,  the third term in~\eqref{eq:taylor-hilbert} (divided by $R_{t-1}$) equals
\begin{align}
    \frac{\int_0^1 (1 - \beta)\varphi_t''(\beta)d\beta}{R_{t-1}} &\leq 
    \cosh \ldba A_t  \rdba \lp \psi_E(\lambda_t) \| Z_t \|^2 +  \ldba \lambda_t Z_t \rdba  \rp^2\label{eq:big_integral}
    \\&\quad \times\int_0^1 (1 - \beta)  \exp \lp \beta \ldba \lambda_t Z_t \rdba  - \beta\psi_E(\lambda_t) \| Z_t \|^2\rp d\beta \nonumber
    \\&= \cosh \ldba A_t  \rdba \lp \psi_E(\lambda_t) \| Z_t \|^2 +  \ldba \lambda_t Z_t \rdba  \rp^2  \nonumber
    \\&\quad\times\frac{  \exp \lp \ldba \lambda_t Z_t \rdba  - \psi_E(\lambda_t) \| Z_t \|^2\rp - \lp \ldba \lambda_t Z_t \rdba  - \psi_E(\lambda_t) \| Z_t \|^2\rp - 1}{\lp \ldba \lambda_t Z_t \rdba  - \psi_E(\lambda_t) \| Z_t \|^2\rp^2} \nonumber
    \\&\stackrel{(i)}{\leq} \cosh \ldba A_t  \rdba \changed{28}{\psi_E(\lambda_t)} \| Z_t \|^2, \nonumber
\end{align}
where (i) follows from Lemma~\ref{lemma:main_ineq} with $D=1$. 

Finally, combining the three pieces of~\eqref{eq:taylor-hilbert} derived above, we get
\begin{align*}
    \frac{\varphi_t(1)}{ R_{t-1}} &= \frac{\varphi_t(0) + \varphi_t'(0) + \int_0^1 (1-\beta)\varphi_t''(\beta)d\beta}{R_{t-1}}
     \\&\leq \cosh \ldba A_t \rdba + \sinh \ldba A_t  \rdba \la \frac{{A_t}}{\left\| A_t  \right\|}, \lambda_t Z_t\ra  - \psi_E(\lambda_t) \| Z_t \|^2\cosh \ldba A_t \rdba 
     \\&\quad+ \psi_E(\lambda_t) \| Z_t \|^2\cosh \ldba A_t \rdba
     \\&= \cosh \ldba A_t \rdba + \sinh \ldba A_t  \rdba \la \frac{{A_t}}{\left\| A_t  \right\|}, \lambda_t Z_t\ra.
\end{align*}

Note that $\delta_{t-1}$  is simply $M_{t-1}$ re-scaled by a positive factor, and so they are both aligned to $A_t = M_{t-1} + \lambda_t \delta_{t-1}$. Thus,
\begin{align*}
    \E_{t-1}\lb \la \frac{A_t}{\left\| A_t  \right\|}, \lambda_t Z_t\ra\rb &= \E_{t-1}\lb \la \frac{A_t}{\left\| A_t  \right\|}, \lambda_t (Y_t - \delta_{t-1})\ra\rb
    \\&= \E_{t-1}\lb \la \frac{A_t}{\left\| A_t  \right\|}, \lambda_t Y_t\ra\rb -  \la \frac{A_t}{\left\| A_t  \right\|}, \lambda_t \delta_{t-1}\ra
    \\&=  \la \frac{A_t}{\left\| A_t \right\|}, \lambda_t \E_{t-1}\lb Y_t\rb\ra - \ldba \lambda_t \delta_{t-1} \rdba 
    \\&= 0 - \ldba \lambda_t \delta_{t-1} \rdba 
    \\&= -\ldba \lambda_t \delta_{t-1} \rdba.
\end{align*}
Furthermore, this also implies that $\ldba M_{t-1}+\lambda_t \delta_{t-1}  \rdba = \ldba M_{t-1}\rdba+\ldba\lambda_t \delta_{t-1} \rdba$. Hence 
\begin{align*}
    -\ldba \lambda_t \delta_{t-1} \rdba\sinh \ldba A_t  \rdba &= -\ldba \lambda_t \delta_{t-1} \rdba\sinh \ldba M_{t-1}+\lambda_t \delta_{t-1}  \rdba
    \\&=  -\ldba \lambda_t \delta_{t-1} \rdba\sinh \lp \ldba M_{t-1}\rdba+\ldba \lambda_t\delta_{t-1}  \rdba\rp    
    \\&\stackrel{(i)}{\leq}  \cosh \ldba M_{t-1}  \rdba -  \cosh \lp \ldba M_{t-1}\rdba+\ldba \lambda_t\delta_{t-1}  \rdba\rp  
    \\&\stackrel{}{=}  \cosh \ldba M_{t-1}  \rdba -  \cosh \ldba M_{t-1}+\lambda_t \delta_{t-1}  \rdba
    \\&\stackrel{}{=}  \cosh \ldba M_{t-1}  \rdba -  \cosh \ldba A_t  \rdba,
\end{align*}
where (i) is obtained in view of Lemma \ref{lemma:ineq_cosh_sinh}.
Recalling our simplifying assumption~\eqref{eq:simplfying-assumption}, we conclude that
\begin{align*}
    \frac{\E_{t-1}[\varphi_t(1)]}{R_{t-1}} &\leq \cosh \ldba A_{t}  \rdba + \cosh \ldba M_{t-1}  \rdba -  \cosh \ldba A_t  \rdba
    \\& = \cosh \ldba M_{t-1}  \rdba, 
\end{align*}
establishing~\eqref{eq:suffice-to-show}, and completing the proof. 

Thus, it only remains to avoid assumption~\eqref{eq:simplfying-assumption}. There are challenges arising from the existence of $\beta^* \in [0, 1]$ such that $A_{t} + \beta^* {\lambda_t} Z_t = 0$, but these can be circumvented. 

First, let us consider the case that $A_t = 0$ (and so $M_{t-1} = \delta_{t-1} = 0$). If so,
\begin{align*}
    \varphi_t(\beta) &= R_{t-1} \cosh \ldba A_{t} + \beta {\lambda_t} Z_t \rdba \exp \lp-\beta\psi_E(\lambda_t) \| Z_t \|^2\rp
    \\&= R_{t-1} \cosh  \lp\beta {\lambda_t} \ldba Y_t \rdba\rp \exp \lp-\beta\psi_E(\lambda_t) \| Y_t \|^2\rp.
\end{align*}
Now observe that the Taylor series of $\varphi_t$ does not require Fréchet differentiability of the norm in this case, since  $\beta$ appears outside of the norm. The arguments provided assuming~\eqref{eq:simplfying-assumption} similarly apply, with no need of Lemma \ref{lemma:ineq_cosh_sinh}, thus completing the proof in this case too. 

We finally consider the case that $A_{t} \neq 0$. In this case, $\beta^*$, if it exists, is unique. We break this case up into two further subcases: when $\beta^* = 1$ and when $\beta^* \in (0,1)$.

If $\beta^* = 1$, then by continuity of $\varphi_t$ on $[0, 1]$,
\begin{align*}
    \varphi_t(1) &= \lim_{\rho \uparrow 1} \varphi_t(\rho)
    \\&= \lim_{\rho \uparrow 1} \lb \varphi_t(0) + \varphi_t'(0) + \int_0^\rho (1-\beta)\varphi_t''(\beta)d\beta \rb
    \\&=  \varphi_t(0) + \varphi_t'(0) + \lim_{\rho \uparrow 1} \lb\int_0^\rho (1-\beta)\varphi_t''(\beta)d\beta \rb.
\end{align*}

It follows that
\begin{align*}
    \int_0^\rho (1 - \beta)\varphi_t''(\beta)d\beta &\leq \cosh \ldba A_t  \rdba \lp \psi_E(\lambda_t) \| Z_t \|^2 +  \ldba \lambda_t Z_t \rdba  \rp^2 
    \\&\quad\times\int_0^\rho (1 - \beta)  \exp \lp \beta \ldba \lambda_t Z_t \rdba  - \beta\psi_E(\lambda_t) \| Z_t \|^2\rp d\beta
    \\&\leq \cosh \ldba A_t  \rdba \lp \psi_E(\lambda_t) \| Z_t \|^2 +  \ldba \lambda_t Z_t \rdba  \rp^2
     \\&\quad\times\int_0^1 (1 - \beta)  \exp \lp \beta \ldba \lambda_t Z_t \rdba  - \beta\psi_E(\lambda_t) \| Z_t \|^2\rp d\beta,
\end{align*}
which is precisely the upper bound in \eqref{eq:big_integral}. Thus the arguments provided assuming~\eqref{eq:simplfying-assumption} also apply to this case. 

It remains to consider the case $\beta^* \in (0,1)$. By continuity of $\varphi_t$,
\begin{align*}
    \varphi_t(1) = \lim_{\rho \downarrow 0} \lb \varphi_t(1) - \varphi_t(\beta^* + \rho) + \varphi_t(\beta^* - \rho) \rb.
\end{align*}
By a Taylor expansion argument,
\begin{align*}
    \varphi_t(1) - \varphi_t(\beta^* + \rho) &= \varphi_t'(\beta^* + \rho)(1 - \beta^* - \rho) + \int_{\beta^* + \rho}^1 (1-\beta) \varphi_t''(\beta) d\beta,
    \\\varphi_t(\beta^* - \rho) &=\varphi_t(0)+ \varphi_t'(0)(\beta^* - \rho) + \int_{0}^{\beta^* - \rho} (\beta^* - \rho-t) \varphi_t''(\beta) d\beta
    \\&= \varphi_t(0)+\varphi_t'(0)(\beta^* - \rho)  \\&\quad+ \int_{0}^{\beta^* - \rho} (1-\beta) \varphi_t''(\beta) d\beta - (1-\beta^* + \rho)\int_{0}^{\beta^* - \rho}  \varphi_t''(\beta) d\beta
    \\&= \varphi_t(0)+\varphi_t'(0)(\beta^* - \rho)  \\&\quad+ \int_{0}^{\beta^* - \rho} (1-\beta) \varphi_t''(\beta) d\beta - (1-\beta^* +\rho) \lp \varphi_t'(\beta^* - \rho) - \varphi_t'(0) \rp
    \\&= \varphi_t(0)+ \varphi_t'(0) + \int_{0}^{\beta^* - \rho} (1-\beta) \varphi_t''(\beta) d\beta - (1-\beta^* +\rho)  \varphi_t'(\beta^* - \rho).
\end{align*}
We remind the reader that, for all $\beta \in [0, 1] \backslash \{\beta^*\}$,
\begin{align*}
\frac{\varphi_t'(\beta)}{R_{t-1}}\exp \lp-
\beta\psi_E(\lambda_t) \| Z_t \|^2\rp&=    \sinh \ldba A_t + \beta \lambda_t Z_t \rdba \la \frac{A_t  + \beta \lambda_t Z_t}{\left\| A_t  + \beta \lambda_t Z_t \right\|}, \lambda_t Z_t\ra  
  \\&\quad - \psi_E(\lambda_t) \| Z_t \|^2\cosh \ldba A_t + \beta \lambda_t Z_t \rdba,
\end{align*}
 so both $\lim_{\beta \downarrow \beta^*}\varphi_t'(\beta) $ and $ \lim_{\beta \uparrow \beta^*}\varphi_t'(\beta)$ coincide, and are equal to
 \begin{align*}
      -R_{t-1} \psi_E(\lambda_t) \| Z_t \|^2\cosh \ldba A_t + \beta^* \lambda_t Z_t \rdba\exp \lp
\beta^*\psi_E(\lambda_t) \| Z_t \|^2\rp,
 \end{align*}
 which implies
\begin{align*}
    \lim_{\rho \downarrow 0} \lb \varphi_t'(\beta^* + \rho)(1 - \beta^* - \rho) - (1-\beta^* +\rho)  \varphi_t'(\beta^* - \rho)\rb = 0.
\end{align*}
It follows that $\varphi_t(1) = \lim_{\rho \downarrow 0} \lb \varphi_t(1) - \varphi_t(\beta^* + \rho) + \varphi_t(\beta^* - \rho) \rb$ is upper bounded by
\begin{align*}
     \limsup_{\rho \downarrow 0} \lb\varphi_t(0)+ \varphi_t'(0) + \int_{0}^{\beta^* - \rho} (1-\beta) \varphi_t''(\beta) d\beta +\int_{\beta^* + \rho}^{1} (1-\beta) \varphi_t''(\beta) d\beta\rb.
\end{align*}
Lastly, $\int_{0}^{\beta^* - \rho} (1-\beta) \varphi_t''(\beta) d\beta +\int_{\beta^* + \rho}^{1} (1-\beta) \varphi_t''(\beta) d\beta$ is again upper bounded by 
\begin{align*}
     \cosh \ldba A_t  \rdba \lp \psi_E(\lambda_t) \| Z_t \|^2 +  \ldba \lambda_t Z_t \rdba  \rp^2 \int_0^1 (1 - \beta)  \exp \lp \beta \ldba \lambda_t Z_t \rdba  - \beta\psi_E(\lambda_t) \| Z_t \|^2\rp d\beta,
\end{align*}
which coincides with \eqref{eq:big_integral}. Thus, the arguments provided for  $A_{t} + \beta {\lambda_t} Z_t \neq 0$ analogously extend to this case.

Thus, having circumvented the assumption~\eqref{eq:simplfying-assumption}, the proof of the theorem is now completed.

\section{Remaining proofs}

\subsection{Proof of Corollary~\ref{cor:empirical_bernstein}} \label{section:proof_main_corollary}

Extend $S_t$ to $t = 0$ with $\lambda_0 = 0$ and $X_0 = 0$. It follows that $S_t$ is a nonnegative supermartingale with $S_0 = 1$. Consequently, Ville's inequality (Fact \ref{fact:villesineq}) yields 
\begin{align*}
    \Pb \lp \sup_{t } S_t \geq \frac{1}{\alpha} \rp \leq \E[S_0] \alpha = \alpha.
\end{align*}

Given that $e^u \leq 2 \cosh(u)$ for all $u \in \R$, it follows from Theorem \ref{theorem:main_theorem} that 
\begin{align*}
    \frac{1}{2}\exp \lp \ldba \sum_{i = 1}^t \lambda_i \frac{X_i- \mu}{4BD} \rdba \rp \exp \lp - \sum_{i = 1}^t \frac{\psi_E(\lambda_i)}{(4B)^2} \| X_i - \bar\mu_{i-1} \|^2  \rp
\end{align*}
is dominated by $S_t$. Thus, with probability $1-\alpha$, and simultaneously for all $t \geq 1$,
\begin{align*}
    \frac{1}{2}\exp \lp \ldba \sum_{i = 1}^t \lambda_i \frac{X_i- \mu}{4BD} \rdba \rp \exp \lp - \sum_{i = 1}^t \frac{\psi_E(\lambda_i)}{(4B)^2} \| X_i - \bar\mu_{i-1} \|^2  \rp \leq \frac{1}{\alpha}.
\end{align*}

Taking logarithms and multiplying both sides by $D\frac{4B}{\sum_{i = 1}^t \lambda_i}$, it follows that
\begin{align*}
        \ldba \frac{\sum_{i=1}^{t} \lambda_i X_i}{\sum_{i=1}^{t} \lambda_i} - \mu \rdba \leq D\frac{\frac{1}{4B} \sum_{i=1}^t \psi_E(\lambda_i) \ldba X_i - \bar\mu_{i-1} \rdba^2 + 4B\log \frac{2}{\alpha}}{\sum_{i=1}^{t} \lambda_i}.
    \end{align*}

\subsection{Proof of Proposition \ref{proposition:radius_ball_iid}} \label{proof:radius_ball_iid}

\begin{changedblock}{41}
\begin{proof}
Because the weights depend on the terminal sample size, it is useful
to make the triangular-array structure explicit. Write
\[
L:=\log(2/\alpha),
\qquad
K:=2(4B)^2L,
\]
and define, recursively for each $n$,
\[
\lambda_{i,n}
:=
\min\left\{
\sqrt{\frac{K}{n\hat\sigma_{i-1,n}^2}},c_1
\right\},
\qquad
\bar\mu_{m,n}
:=
\frac{\sum_{i=1}^m\lambda_{i,n}X_i}
{\sum_{i=1}^m\lambda_{i,n}},
\]
where
\[
\hat\sigma_{m,n}^2
:=
\frac{c_2B^2+\sum_{i=1}^m
\norm{X_i-\bar\mu_{i-1,n}}^2}{m+1},
\qquad
\bar\mu_{0,n}:=0.
\]
Set
\[
Z_{i,n}:=\norm{X_i-\bar\mu_{i-1,n}}^2.
\]

The proof proceeds in five steps. We first show that, after a logarithmic burn-in, the empirical variance remains bounded away from zero. This controls the adaptive weights and allows us to prove a uniform weighted strong law. We then use that strong law to establish uniform consistency of the empirical variance, identify the asymptotic behavior of the weights on the bulk of the sample, and finally show that the omitted initial indices are negligible.

\medskip
\noindent
\textit{Step 1: A lower bound for the empirical variance.}

The purpose of this step is to show that, after a logarithmic
burn-in, the adaptive variance estimate does not become too small.
This will provide a deterministic $O(n^{-1/2})$ upper bound on the
weights.

Let $\F_i:=\sigma(X_1,\ldots,X_i)$. Since
$\bar\mu_{i-1,n}$ is a convex combination of
$X_1,\ldots,X_{i-1}$, and $\bar\mu_{0,n}=0$, we have
\[
\norm{\bar\mu_{i-1,n}}\le B.
\]
More generally, for every $\F_{i-1}$-measurable $y$
satisfying $\norm{y}\le B$, Jensen's inequality gives
\[
\norm{y-\mu}^2
=
\norm{\E_{i-1}(y-X_i)}^2
\le
\E_{i-1}\norm{X_i-y}^2.
\]
It follows that
\[
\begin{aligned}
\sigma^2
&=
\E_{i-1}\norm{X_i-\mu}^2\\
&\le
2\E_{i-1}\norm{X_i-y}^2
+2\norm{y-\mu}^2\\
&\le
4\E_{i-1}\norm{X_i-y}^2.
\end{aligned}
\]
Taking $y=\bar\mu_{i-1,n}$, we obtain
\[
\E_{i-1}Z_{i,n}\ge\frac{\sigma^2}{4},
\qquad
0\le Z_{i,n}\le4B^2.
\]

Let
\[
\Delta_{i,n}:=
Z_{i,n}-\E_{i-1}Z_{i,n}.
\]
Then $(\Delta_{i,n})_{i\le n}$ is a scalar martingale-difference
sequence with $|\Delta_{i,n}|\le4B^2$. Therefore, for every $m\le n$,
Azuma--Hoeffding gives
\[
\begin{aligned}
\Pb\left(
\sum_{i=1}^mZ_{i,n}<\frac{m\sigma^2}{8}
\right)
&\le
\Pb\left(
\sum_{i=1}^m\Delta_{i,n}<-\frac{m\sigma^2}{8}
\right)\\
&\le e^{-c_0m},
\end{aligned}
\]
where
\[
c_0:=\frac{\sigma^4}{2048B^4}.
\]

Choose
\[
q_n:=
\left\lceil\frac{3}{c_0}\log(n+1)\right\rceil.
\]
The choice $q_n\asymp\log n$ makes the probability of failure
summable while ensuring that the initial block remains negligible.
Indeed, a union bound yields
\[
\begin{aligned}
&\Pb\left(
\exists m\in\{q_n,\ldots,n\}:
\sum_{i=1}^mZ_{i,n}<\frac{m\sigma^2}{8}
\right)\\
&\qquad\le
\sum_{m=q_n}^ne^{-c_0m}
\le
\frac{e^{-c_0q_n}}{1-e^{-c_0}}
\le C(n+1)^{-3}.
\end{aligned}
\]
Since this bound is summable, Borel--Cantelli implies that, almost
surely, for every sufficiently large $n$ and every
$q_n\le m\le n$,
\begin{equation}
\hat\sigma_{m,n}^2
\ge
\frac{m\sigma^2}{8(m+1)}
\ge
\frac{\sigma^2}{16}.
\label{eq:variance-lower-bound}
\end{equation}

We also have the deterministic upper bound
\[
\hat\sigma_{m,n}^2
\le
\frac{B^2+4mB^2}{m+1}
\le4B^2.
\]
Consequently, for all sufficiently large $n$,
\[
\lambda_{i,n}\ge\frac{a}{\sqrt n},
\qquad
a:=\frac{\sqrt K}{2B},
\]
for every $i$. On the event in
\eqref{eq:variance-lower-bound}, we also have
\[
\lambda_{i,n}\le\frac{b}{\sqrt n},
\qquad
b:=\frac{4\sqrt K}{\sigma},
\qquad i>q_n.
\]

\medskip
\noindent
\textit{Step 2: An adaptive weighted strong law.}

We now prove that, for every fixed $\eta\in(0,1)$,
\begin{equation}
\max_{\eta n\le m\le n}
\norm{\bar\mu_{m,n}-\mu}
\longrightarrow0
\qquad\text{a.s.}
\label{eq:weighted-slln}
\end{equation}
The idea is to write
\[
\bar\mu_{m,n}-\mu
=
\frac{
\sum_{i=1}^m\lambda_{i,n}(X_i-\mu)
}{
\sum_{i=1}^m\lambda_{i,n}
}.
\]
We will show that the numerator is $o(\sqrt n)$, uniformly over
$m$, whereas the denominator is bounded below by a constant
multiple of $\sqrt n$.

The upper bound $\lambda_{i,n}\le b/\sqrt n$ obtained above only
holds on the eventual event \eqref{eq:variance-lower-bound}. To
apply a martingale inequality unconditionally, we stop the weights
at the first failure of that event. Define
\[
\tau_n
:=
\inf\left\{
m\in\{q_n,\ldots,n\}:
\hat\sigma_{m,n}^2<\frac{\sigma^2}{16}
\right\},
\]
with $\inf\varnothing=n+1$, and set
\[
\widetilde\lambda_{i,n}
:=
\lambda_{i,n}\mathbf 1\{\tau_n>i-1\},
\qquad i>q_n.
\]
The use of the first failure time ensures predictability, since
\[
\{\tau_n>i-1\}
=
\bigcap_{m=q_n}^{i-1}
\left\{
\hat\sigma_{m,n}^2\ge\frac{\sigma^2}{16}
\right\}
\in\F_{i-1}.
\]
Moreover, on every sample path,
\[
\norm{\widetilde\lambda_{i,n}(X_i-\mu)}
\le
\frac{2Bb}{\sqrt n}.
\]

The maximal Hoeffding inequality for martingale differences in a
$(2,D)$-smooth Banach space \cite[Theorem~3.1]{pinelis1994optimum} therefore gives, for every
$\epsilon>0$,
\[
\begin{aligned}
&\Pb\left(
\max_{q_n<m\le n}
\left\|
\sum_{i=q_n+1}^m
\widetilde\lambda_{i,n}(X_i-\mu)
\right\|
>\epsilon\sqrt n
\right)\\
&\qquad\le
2\exp\left(
-\frac{\epsilon^2n}{8D^2B^2b^2}
\right).
\end{aligned}
\]
On the event \eqref{eq:variance-lower-bound}, no stopping occurs.
Thus
\[
\begin{aligned}
&\Pb\left(
\max_{q_n<m\le n}
\left\|
\sum_{i=q_n+1}^m
\lambda_{i,n}(X_i-\mu)
\right\|
>\epsilon\sqrt n
\right)\\
&\qquad\le
C(n+1)^{-3}
+
2\exp\left(
-\frac{\epsilon^2n}{8D^2B^2b^2}
\right).
\end{aligned}
\]
The right-hand side is summable in $n$. Applying
Borel--Cantelli for $\epsilon=1/k$, $k\in\mathbb N$, and
intersecting the resulting probability-one events, we obtain
\begin{equation}
\frac1{\sqrt n}
\max_{q_n<m\le n}
\left\|
\sum_{i=q_n+1}^m
\lambda_{i,n}(X_i-\mu)
\right\|
\longrightarrow0
\qquad\text{a.s.}
\label{eq:martingale-transform}
\end{equation}

It remains to check that the logarithmic burn-in is negligible.
For $i\le q_n$,
\[
\hat\sigma_{i-1,n}^2
\ge
\frac{c_2B^2}{i},
\qquad
\lambda_{i,n}
\le
\sqrt{\frac{Ki}{c_2B^2n}}.
\]
Since $q_n=O(\log n)$,
\[
\begin{aligned}
\left\|
\sum_{i=1}^{q_n}\lambda_{i,n}(X_i-\mu)
\right\|
&\le
2B\sqrt{\frac{K}{c_2B^2n}}
\sum_{i=1}^{q_n}\sqrt i\\
&=
O\left(\frac{q_n^{3/2}}{\sqrt n}\right)
=o(\sqrt n).
\end{aligned}
\]
On the other hand, for every $m\ge\eta n$,
\[
\sum_{i=1}^m\lambda_{i,n}
\ge
\frac{am}{\sqrt n}
\ge
a\eta\sqrt n.
\]
Combining these bounds with
\eqref{eq:martingale-transform} proves
\eqref{eq:weighted-slln}.

\medskip
\noindent
\textit{Step 3: Uniform consistency of the empirical variance.}

The adaptive weighted SLLN controls the error introduced by
centering each observation at $\bar\mu_{i-1,n}$ rather than at
$\mu$. We first show that this centering error is negligible on
average.

Fix an outer parameter $\eta\in(0,1)$, and write
\[
d_{j,n}:=\norm{\bar\mu_{j,n}-\mu}.
\]
For an auxiliary cutoff $0<\delta<\eta$ and $m\ge\eta n$,
split the average into its initial and bulk portions:
\[
\begin{aligned}
\frac1m\sum_{i=1}^m d_{i-1,n}
&\le
\frac{2B\lfloor\delta n\rfloor}{\eta n}
+
\max_{\lfloor\delta n\rfloor\le j\le n}d_{j,n}\\
&\le
\frac{2B\delta}{\eta}
+
\max_{\lfloor\delta n\rfloor\le j\le n}d_{j,n}.
\end{aligned}
\]
The first term corresponds to a small initial fraction of the
sample, on which we only use the crude bound $d_{j,n}\le2B$. For
the remaining indices, set
\[
\rho:=\frac{\delta}{2}.
\]
For all sufficiently large $n$,
$\lfloor\delta n\rfloor\ge\rho n$, and hence
\eqref{eq:weighted-slln}, applied with the auxiliary parameter
$\rho$, gives
\[
\max_{\lfloor\delta n\rfloor\le j\le n}d_{j,n}
\le
\max_{\rho n\le j\le n}d_{j,n}
\longrightarrow0
\qquad\text{a.s.}
\]
Therefore,
\[
\limsup_{n\to\infty}
\max_{\eta n\le m\le n}
\frac1m\sum_{i=1}^m d_{i-1,n}
\le
\frac{2B\delta}{\eta}.
\]
Here $\eta$ remains fixed. Taking the countable sequence
$\delta_k=\eta2^{-k}$, intersecting the corresponding
probability-one events, and then letting $k\to\infty$, we obtain
\begin{equation}
\max_{\eta n\le m\le n}
\frac1m\sum_{i=1}^m
\norm{\bar\mu_{i-1,n}-\mu}
\longrightarrow0
\qquad\text{a.s.}
\label{eq:average-centering-error}
\end{equation}

Now let
\[
Y_i:=\norm{X_i-\mu}^2.
\]
The ordinary scalar SLLN gives
\[
\frac1m\sum_{i=1}^mY_i\longrightarrow\sigma^2
\qquad\text{a.s.}
\]
A convergent sequence has uniformly small tails, so
\begin{equation}
\max_{\eta n\le m\le n}
\left|
\frac1m\sum_{i=1}^mY_i-\sigma^2
\right|
\longrightarrow0
\qquad\text{a.s.}
\label{eq:uniform-scalar-slln}
\end{equation}

Furthermore,
\[
\begin{aligned}
|Z_{i,n}-Y_i|
&=
\left|
\norm{X_i-\bar\mu_{i-1,n}}^2
-\norm{X_i-\mu}^2
\right|\\
&\le
4B\norm{\bar\mu_{i-1,n}-\mu}.
\end{aligned}
\]
Combining \eqref{eq:average-centering-error} and
\eqref{eq:uniform-scalar-slln}, we obtain
\[
\max_{\eta n\le m\le n}
\left|
\frac1m\sum_{i=1}^mZ_{i,n}-\sigma^2
\right|
\longrightarrow0
\qquad\text{a.s.}
\]
Finally,
\[
\hat\sigma_{m,n}^2
=
\frac{m}{m+1}
\left(\frac1m\sum_{i=1}^mZ_{i,n}\right)
+
\frac{c_2B^2}{m+1}.
\]
Since $m\ge\eta n$, the regularization and denominator corrections
are uniformly $O(n^{-1})$. Hence, for every fixed
$\eta\in(0,1)$,
\begin{equation}
\max_{\eta n\le m\le n}
|\hat\sigma_{m,n}^2-\sigma^2|
\longrightarrow0
\qquad\text{a.s.}
\label{eq:uniform-variance-consistency}
\end{equation}

\medskip
\noindent
\textit{Step 4: Asymptotic behavior of the weights on the bulk.}

Fix $\eta\in(0,1)$. For all sufficiently large $n$,
$i-1\ge\eta n/2$ whenever $i\ge\lceil\eta n\rceil$.
Applying \eqref{eq:uniform-variance-consistency} with parameter
$\eta/2$ gives
\[
\max_{\lceil\eta n\rceil\le i\le n}
|\hat\sigma_{i-1,n}^2-\sigma^2|
\longrightarrow0
\qquad\text{a.s.}
\]
Since $\sigma>0$, the truncation at $c_1$ is therefore uniformly
inactive on this range for all sufficiently large $n$. Thus
\[
\lambda_{i,n}
=
\sqrt{\frac{K}{n\hat\sigma_{i-1,n}^2}},
\]
and continuity of $x\mapsto x^{-1/2}$ gives
\begin{equation}
\max_{\lceil\eta n\rceil\le i\le n}
\left|
\sqrt n\,\lambda_{i,n}
-\frac{\sqrt K}{\sigma}
\right|
\longrightarrow0
\qquad\text{a.s.}
\label{eq:bulk-lambda-limit}
\end{equation}

Also, $\lambda_{i,n}\to0$ uniformly on this range. Since
\[
\frac{\psi_E(u)}{u^2}\longrightarrow\frac12
\qquad (u\downarrow0),
\]
we obtain
\begin{equation}
\max_{\lceil\eta n\rceil\le i\le n}
\left|
n\psi_E(\lambda_{i,n})
-\frac{K}{2\sigma^2}
\right|
\longrightarrow0
\qquad\text{a.s.}
\label{eq:bulk-psi-limit}
\end{equation}

Let $r_n:=\lceil\eta n\rceil$. The definition of
$\hat\sigma_{m,n}^2$ gives the exact telescoping identity
\[
\frac1n\sum_{i=r_n}^nZ_{i,n}
=
\frac{
(n+1)\hat\sigma_{n,n}^2
-r_n\hat\sigma_{r_n-1,n}^2
}{n}.
\]
Using \eqref{eq:uniform-variance-consistency} at $m=n$ and
$m=r_n-1$, we conclude that
\begin{equation}
\frac1n\sum_{i=r_n}^nZ_{i,n}
\longrightarrow
(1-\eta)\sigma^2
\qquad\text{a.s.}
\label{eq:bulk-Z-limit}
\end{equation}

Equations \eqref{eq:bulk-psi-limit} and
\eqref{eq:bulk-Z-limit} now imply
\[
\begin{aligned}
\sum_{i=r_n}^n\psi_E(\lambda_{i,n})Z_{i,n}
&=
\frac1n\sum_{i=r_n}^n
\big[n\psi_E(\lambda_{i,n})\big]Z_{i,n}\\
&\longrightarrow
\frac{K}{2}(1-\eta)
\end{aligned}
\]
almost surely. Similarly, \eqref{eq:bulk-lambda-limit} gives
\[
\frac1{\sqrt n}\sum_{i=r_n}^n\lambda_{i,n}
=
\frac1n\sum_{i=r_n}^n
\big[\sqrt n\,\lambda_{i,n}\big]
\longrightarrow
(1-\eta)\frac{\sqrt K}{\sigma}
\]
almost surely.

\medskip
\noindent
\textit{Step 5: Removing the initial indices and evaluating the
limit.}

It remains to show that the indices preceding $\eta n$ can be
discarded as $\eta\downarrow0$. Let
\[
C_\psi
:=
\sup_{0<u\le c_1}\frac{\psi_E(u)}{u^2}
<\infty.
\]

For the logarithmic burn-in $i\le q_n$, the bounds already
established give
\[
\sum_{i\le q_n}\psi_E(\lambda_{i,n})Z_{i,n}
=
O\left(\frac{q_n^2}{n}\right)
=o(1)
\]
and
\[
\frac1{\sqrt n}\sum_{i\le q_n}\lambda_{i,n}
=
O\left(\frac{q_n^{3/2}}{n}\right)
=o(1).
\]

On the eventual event \eqref{eq:variance-lower-bound}, the remaining
indices $q_n<i<\eta n$ satisfy
\[
\lambda_{i,n}\le\frac b{\sqrt n}.
\]
Consequently,
\[
\sum_{q_n<i<\eta n}
\psi_E(\lambda_{i,n})Z_{i,n}
\le
4B^2C_\psi b^2\eta
\]
and
\[
\frac1{\sqrt n}
\sum_{q_n<i<\eta n}\lambda_{i,n}
\le b\eta.
\]

Because all these terms are nonnegative, the bulk limits and the
preceding upper bounds imply
\[
\frac K2(1-\eta)
\le
\liminf_{n\to\infty}
\sum_{i=1}^n\psi_E(\lambda_{i,n})Z_{i,n}
\]
and
\[
\limsup_{n\to\infty}
\sum_{i=1}^n\psi_E(\lambda_{i,n})Z_{i,n}
\le
\frac K2(1-\eta)+4B^2C_\psi b^2\eta.
\]
Letting $\eta\downarrow0$ yields
\[
\sum_{i=1}^n\psi_E(\lambda_{i,n})Z_{i,n}
\longrightarrow\frac K2
\qquad\text{a.s.}
\]
The same argument gives
\[
\frac1{\sqrt n}\sum_{i=1}^n\lambda_{i,n}
\longrightarrow
\frac{\sqrt K}{\sigma}
\qquad\text{a.s.}
\]

We can now evaluate the confidence radius. Multiplying it by
$\sqrt n$ and using the two preceding limits gives
\[
\begin{aligned}
&\sqrt n\,D
\frac{
\frac1{4B}
\sum_{i=1}^n\psi_E(\lambda_{i,n})Z_{i,n}
+4BL
}{
\sum_{i=1}^n\lambda_{i,n}
}\\
&\qquad\longrightarrow
D\,
\frac{\frac{K}{8B}+4BL}{\sqrt K/\sigma}.
\end{aligned}
\]
Since $K=2(4B)^2L$, the last expression equals
\[
\sigma D\sqrt{2L}
=
\sigma D\sqrt{2\log(2/\alpha)}.
\]
This proves the proposition.
\end{proof}
\end{changedblock}

\section{Properties of Pinelis' Bernstein-type inequality}

\subsection{Derivation of Pinelis' Bernstein-type inequality} \label{appendix:pinelis_derivation}

Let $X_1, X_2, \ldots$ be such that $\E_{t-1}\lb X_t\rb = \mu$ and $\|X_t - \mu \| \leq B$. \cite{pinelis1994optimum}, Theorem 3.1, establishes that the probability $\Pb \lp  \sup_t \ldba \sum_{i = 1}^t (X_i- \mu) \rdba \geq r \rp$ is upper bounded by 
\begin{align*}
     &2 \exp \lp - \lambda r + D^2 \ldba \sum_{i=1}^t E_{i-1}\lb e^{\lambda \| X_i-\mu\|} - 1 - \lambda \ldba X_i-\mu \rdba\rb \rdba_\infty\rp,
\end{align*}
where $r \geq 0$ and $\lambda > 0$ are arbitrary. The upper bound
\begin{align*}
    E_{i-1}\lb \sum_{k \geq 2} \frac{\lp \lambda \| X-\mu\| \rp^k}{k!} \rb &\leq  \frac{E_{i-1}\lb \| X-\mu\|^2\rb}{B^2} \sum_{k \geq 2} \frac{\lp \lambda B\rp^k}{k!} 
    \\&=E_{i-1}\lb \| X-\mu\|^2\rb\frac{e^{\lambda B} - 1 - \lambda B}{B^2}
\end{align*}
yields that, for all $r \geq 0$, $\Pb \lp  \sup_t\ldba \sum_{i = 1}^t (X_i- \mu) \rdba \geq r \rp$ is upper bounded by
\begin{align} \label{eq:bernstein_type_ineq_yet_to_optimize}
     2 \exp \lp - \lambda r + D^2\ldba \sum_{i=1}^t E_{i-1}\lb \| X-\mu\|^2\rb \rdba_\infty\frac{e^{\lambda B} - 1 - \lambda B}{B^2}\rp,
\end{align}
where $\lambda > 0$ is arbitrary (this is, precisely, the Bennett-type inequality exhibited in \cite{pinelis1994optimum}, Theorem 3.4). 

For a fixed $t$, the  optimization of $\lambda$ for minimizing the right hand side of \eqref{eq:bernstein_type_ineq_yet_to_optimize} leads to the usual Bennett and Bernstein type confidence intervals in the scalar setting with
\begin{align*}
    V:=D^2\ldba \sum_{i=1}^t E_{i-1}\lb \| X-\mu\|^2\rb \rdba_\infty \text{ instead of } \sigma^2t.
\end{align*}

Taking the derivative with respect to $\lambda$ of the exponent of the right hand side of \eqref{eq:bernstein_type_ineq_yet_to_optimize} and setting it to zero, we obtain 
\begin{align*}
    \lambda^* = \frac{1}{B} \log \lp 1 + \frac{Br}{V} \rp.
\end{align*}
Pluging $\lambda^*$ on the right hand side of \eqref{eq:bernstein_type_ineq_yet_to_optimize} we obtain the upper bound 
\begin{align*}
    -\frac{V}{B^2} h\lp \frac{rB}{V} \rp,
\end{align*}
where $h(u) = (1+u)\log (1+u) - u$, $u \geq 0$. It can be shown that
\begin{align*}
    h(u) \geq \frac{u^2}{2(1 + \frac{u}{3})},
\end{align*}
and so 
\begin{align}  \label{eq:oracle_bernstein_prob}
    \Pb \lp  \ldba \sum_{i = 1}^t (X_i- \mu) \rdba \geq r \rp \leq 2 \exp \lp - \frac{r^2}{2(V + r\frac{B}{3})}\rp, \quad \forall r \geq 0.
\end{align}

\subsection{Limiting Radius of Pinelis' Bernstein-type confidence intervals} \label{appendix:radius_pinelis}
The arguments we present for deriving the limiting radius of the Bernstein-type confidence sets from \cite{pinelis1994optimum} are analogous to the well studied scalar problem; we derive them in here for completeness of the work. Let $X_1, \ldots, X_n$ be iid with $\E\lb X\rb = \mu$,  $\sigma^2 = \E\lb \| X - \mu \|^2\rb$, and $\|X - \mu \| \leq B$. Taking $V = nD^2 \sigma^2$ in \eqref{eq:oracle_bernstein_prob} and solving for 
\begin{align*}
    2 \exp \lp - \frac{r^2}{2(V + r\frac{B}{3})}\rp = \alpha
\end{align*}
leads to the nonnegative solution
\begin{align*}
    r = \frac{V \log (2 / \alpha)}{3} + \sqrt{\frac{B^2 \log^2(2 / \alpha)}{9} + 2V\log(2/\alpha)} \leq \sqrt{2V\log(2/\alpha)} + \frac{2B\log(2/\alpha)}{3}.
\end{align*}
Consequently,
\begin{align*} 
    \Pb \lp  \ldba \sum_{i = 1}^n (X_i- \mu) \rdba \geq \sqrt{2V\log(2/\alpha)} + \frac{2B\log(2/\alpha)}{3} \rp \leq \alpha,
\end{align*}
which implies
\begin{align*} 
    \Pb \lp  \sqrt{n}\ldba \frac{1}{n} \sum_{i = 1}^n (X_i- \mu) \rdba \geq \sigma D \sqrt{2\log(2/\alpha)} + \frac{2B\log(2/\alpha)}{3\sqrt{n}} \rp \leq \alpha.
\end{align*}
Note that the limiting width scales as $\sigma D \sqrt{2\log(2/\alpha)}$, precisely the same rate as established in Proposition~\ref{proposition:radius_ball_iid}.

\subsection{On the optimality of Pinelis'  Bernstein-type inequality} \label{appendix:optimality_pinelis}

We compare here the oracle Bernstein type inequalities from \cite{pinelis1994optimum} to those from \cite{kohler2017sub} and \cite{gross2011recovering}.
Let $n$ be fixed and $X_1, \ldots, X_n$ be iid with $\E\lb X\rb = \mu$,  $\sigma^2 = \E\lb \| X - \mu \|^2\rb$, and $\|X\| \leq B$. Following the limitations of \cite{gross2011recovering, kohler2017sub}, we assume that the data belongs to a separable Hilbert space, and so the bounds from \cite{pinelis1994optimum} are taken with smoothness parameter $D = 1$.

As derived on Appendix \ref{appendix:pinelis_derivation}, the results from \cite{pinelis1994optimum} lead to   
\begin{align*} 
    \Pb \lp  \ldba \sum_{i = 1}^n (X_i- \mu) \rdba \geq r \rp \leq 2 \exp \lp - \frac{r^2}{2(n \sigma^2 + r\frac{B}{3})}\rp, \quad \forall r \geq 0.
\end{align*}
Thus
\begin{align} \label{eq:pinelis_bernstein_oracle_prob}
    \Pb \lp  \ldba \frac{1}{n}\sum_{i = 1}^n (X_i- \mu) \rdba \geq r \rp \leq 2 \exp \lp - \frac{n r^2}{2( \sigma^2 + r\frac{B}{3n})}\rp, \quad \forall r \geq 0.
\end{align}

In contrast, the following bound may be found in \cite{kohler2017sub}, Lemma 18, which is an extension of \cite{gross2011recovering}, Theorem 12, to the average of independent random vectors:
\begin{align} \label{eq:kohler_bernstein_oracle_prob}
     \Pb \lp  \ldba \frac{1}{n}\sum_{i = 1}^n (X_i- \mu) \rdba \geq r \rp \leq  \exp \lp - \frac{n r^2}{8\sigma^2} + \frac{1}{4}\rp, \quad \forall r \in \lp 0, \frac{\sigma^2}{B}  \rp.
\end{align}

Note that $r$ is freely chosen among positive values in \eqref{eq:pinelis_bernstein_oracle_prob}, while it is restricted to values smaller than $\sigma^2 / B$ in \eqref{eq:kohler_bernstein_oracle_prob}. Most importantly, for those $r < \sigma^2 / B$, 
\begin{align*}
    2 \exp \lp - \frac{n r^2}{2( \sigma^2 + r\frac{B}{3n})}\rp \stackrel{(i)}{<} 2 \exp \lp - \frac{n r^2}{2\sigma^2(1 +\frac{1}{3})} \rp = 2 \exp \lp - \frac{3n r^2}{8\sigma^2} \rp,
\end{align*}
where (i) is substantially loose for big $n$. Thus, as soon as 
\begin{align} \label{eq:easy_condition}
    \exp\lp - \frac{n r^2}{8\sigma^2} \rp \leq \sqrt{\frac{1}{2}\exp \lp \frac{1}{4}\rp} \approx 0.80,
\end{align}
it holds that 
\begin{align*}
    2 \exp \lp - \frac{n r^2}{2( \sigma^2 + r\frac{B}{3n})}\rp \stackrel{}{<} \exp \lp - \frac{n r^2}{8\sigma^2} + \frac{1}{4}\rp, 
\end{align*}
that is, the bound from \cite{pinelis1994optimum} is tighter. If \eqref{eq:easy_condition} does not hold, then the bound from \cite{kohler2017sub} is not informative to begin with since the above right hand side is larger than 1. 

Furthermore, we noted that the inequality (i) was not tight for big $n$. The $(1-\alpha)$ confidence interval derived from \eqref{eq:kohler_bernstein_oracle_prob} is of the form
\begin{align*}
     \sqrt{n}\ldba \frac{1}{n}\sum_{i = 1}^n (X_i- \mu) \rdba \leq \sigma \sqrt{8\lp \log\lp \frac{1}{\alpha} \rp + \frac{1}{4} \rp}.
\end{align*}
The width of such a confidence interval is clearly greater than  $\sigma \sqrt{2\log(2/\alpha)}$, the limiting width of \cite{pinelis1994optimum} confidence interval derived in Appendix \ref{appendix:radius_pinelis}. 
\section{Deriving confidence sequences that achieve LIL rates} \label{appendix:lil}

We devote this section to presenting how Theorem \ref{theorem:main_theorem} yields confidence sequences whose order matches that implied by the law of the iterated logarithm (LIL), both asymptotically and in the finite sample regime. 

\textbf{Asymptotic upper LIL.}
Let us start by defining the functions
  \begin{align*}
      \psi_{E, c}(\lambda) := \frac{-\log(1 - c\lambda) - c\lambda}{c^2}, \quad \psi_{G,c}(\lambda) := \frac{\lambda^2}{2 (1 - \lambda)}.
  \end{align*}
  Based on $e^x \leq 2 \cosh x$, Theorem~\ref{theorem:main_theorem} implies that 
   \begin{align} \nonumber
        \tilde S_t =  \exp \lp \lambda \ldba \sum_{i = 1}^t  \frac{X_i- \mu}{D} \rdba- \sum_{i = 1}^t \psi_{E, 4B}(\lambda) \| X_i - \bar\mu_{i-1} \|^2  \rp, \quad \lambda \in \left(0, \frac{0.8}{4B}\right],
    \end{align}
    is dominated by the nonnegative supermartingale $2S_t$. This means that  $\ldba M_t \rdba$ is $2$-sub-$\psi_{E, 4B}$ with variance process $V_t$ (see \cite{howard2021time}, Definition 1).

\begin{proof} [Proof of Corollary~\ref{corollary:lil}]
Given that $\ldba M_t \rdba$ is $2$-sub-$\psi_{E, 4B}$ with variance process $V_t$, applying~\cite{howard2021time}, Corollary 1, concludes the proof after noting that $\psi_{E, 4B} \approx \lambda^2 / 2$ as $\lambda \downarrow 0$.
\end{proof}


\textbf{Finite LIL bound.} In the finite sample regime, it is still possible to construct confidence sequences that scale as $O\lp \sqrt{V_t \log\log V_t}\rp$. This can be done by repeatedly applying Theorem~\ref{theorem:main_theorem} over geometrically-spaced epochs in time (this technique is often referred to as stitching, peeling, or chaining). More specifically, it is known that $\psi_{E, c}(\lambda) \leq \psi_{G,c}(\lambda)$ on $\lambda \in [0, 1/c)$, so $\ldba M_t \rdba$ is also $2$-sub-$\psi_{G, 4B}$ with variance process $V_t$. \cite{howard2021time}, Theorem 1, may thus be invoked to obtain a finite LIL bound, with $l_0 = \E[2S_0]=2$. The  hyperparameters $\eta$, $m$, and $h(k)$ in \cite{howard2021time}, Theorem 1, should be chosen carefully (depending on the values  $\alpha$ and $B$) in order  to provide sensible bounds.

\begin{proof} [Proof of Corollary~\ref{corollary:finite_lil}]
    Given that $\ldba M_t \rdba$ is $2$-sub-$\psi_{E, 4B}$ with variance process $V_t$, applying~\cite{howard2021time}, Theorem 1, with $l_0 = 2$
    and hyperparameters $\eta = 2$, $m = 1$, and $h(k) \propto k^{1.4}$, yields \eqref{eq:finite_lil_bound}.
\end{proof}

 Note that~\eqref{eq:finite_lil_bound} is analogous to \cite{howard2021time}, Equation (24), inflated by an extra factor $D$ which dictates the (lack of) smoothness of the Banach space.

\section*{Acknowledgments}
The authors would like to thank Ben Chugg and Martin Larsson for insightful conversations. Correspondence concerning this article should be addressed to the first author.

\section*{Funding}
\ManuscriptFunding


\nocite{waudby2020confidence}
\bibliographystyle{unsrtnat}
\bibliography{files/bibliography}

\end{document}